\documentclass[a4paper,american,11pt]{amsart}

\usepackage[latin1]{inputenc}
\usepackage[american]{babel}
\usepackage{amsmath}
\usepackage{amssymb}
\usepackage{amsfonts}
\usepackage{amsthm}
\usepackage{tikz-cd}
\usepackage{mathtools}

\usepackage{graphicx}
\usepackage{xcolor}
\usepackage{url}
\usepackage{bbm}
\usepackage{enumitem}
\usepackage[margin = 3cm]{geometry}
\usepackage{setspace}
\usepackage[normalem]{ulem}
\usepackage[colorlinks=true]{hyperref}
\usepackage[textwidth=30mm]{todonotes}

\newcommand{\ZZ}{{\mathbb Z}}
\newcommand{\Z}{{\mathbb Z}}

\newcommand{\CC}{{\mathbb C}}
\newcommand{\RR}{{\mathbb R}}
\newcommand{\QQ}{{\mathbb Q}}

\newcommand{\TT}{{\mathbb T}}
\newcommand{\f}{\Sigma}
\newcommand{\XX}{{\mathcal X}}

\newcommand{\Vertices}{{\mathcal V}}
\newcommand{\Edges}{{\mathcal E}}
\newcommand{\Faces}{{\mathcal F}}

\newcommand{\bfw}{\mathbf w}

\newcommand{\bfe}{\mathbf e}

\newcommand{\B}{{\mathcal B}}

\newcommand{\Fs}{{\mathbf F}}

\newcommand{\A}{{\mathcal A}}
\renewcommand{\L}{{\mathcal L}}

\newcommand{\F}{\mathcal F}

\newcommand{\Rbar}{\overline{\mathbb{R}}}

\newcommand{\comment}[1]{}

\DeclareMathOperator{\rec}{rec}
\DeclareMathOperator{\In}{in}
\DeclareMathOperator{\trop}{Trop}

\DeclareMathOperator{\MW}{MW}

\DeclareMathOperator{\Aff}{Aff}
\DeclareMathOperator{\csm}{csm}
\DeclareMathOperator{\Hom}{Hom}
\DeclareMathOperator{\Pic}{Pic}

\DeclareMathOperator{\Div}{div}

\DeclareMathOperator{\cone}{cone}
\DeclareMathOperator{\spann}{span}
\DeclareMathOperator{\starr}{star}

\DeclareMathOperator{\divis}{div}
\DeclareMathOperator{\Ac}{A}
\DeclareMathOperator{\sed}{sed}

\DeclareMathOperator{\Todd}{Todd}

\newcommand{\TP}{\mathbb{TP}}

\newcommand{\T}{\mathbb{T}}
\newcommand{\R}{\mathbb{R}}
\newcommand{\C}{\mathbb{C}}

\newtheorem{thm}{Theorem}[section]
\newtheorem{prop}[thm]{Proposition}

\newtheorem{lemma}[thm]{Lemma}
\newtheorem{cor}[thm]{Corollary}

\newtheorem{conj}[thm]{Conjecture}

         {\theoremstyle{definition}

\newtheorem{exa*}[thm]{Example}
\newtheorem{defi}[thm]{Definition}
\newtheorem{rem}[thm]{Remark}}

\DeclareRobustCommand{\qedify}[1]{%
  \ifmmode \quad\hbox{#1}
  \else
    \leavevmode\unskip\penalty9999 \hbox{}\nobreak\hfill
    \quad\hbox{#1}%
  \fi
  }
\newenvironment{exa}{\begin{exa*}\pushQED{\qedify{$\diamondsuit$}}}{\popQED\end{exa*}}

\makeatletter
\newtheorem*{rep@theorem}{\rep@title}
\newcommand{\newreptheorem}[2]{%
\newenvironment{rep#1}[1]{%
 \def\rep@title{{\bf #2 \ref{##1}}}%
 \begin{rep@theorem}}%
 {\end{rep@theorem}}}
\makeatother

\newreptheorem{theorem}{Theorem}

\numberwithin{equation}{section}

\begin{document}

\title{Chern classes of tropical manifolds}


\author{ Luc\'ia L\'opez de Medrano}
\address{Unidad Cuernavaca del Instituto de Matem\'aticas, Universidad Nacional Aut\'onoma de M\'exico. Cuernavaca, M\'exico.}
\email{lucia.ldm@im.unam.mx}

\author{Felipe Rinc\'on}
\address{School of Mathematical Sciences, Queen Mary University of London, London, United Kingdom.}
\email{f.rincon@qmul.ac.uk}

\author{Kris Shaw} 
\address{Department of Mathematics, University of Oslo, Oslo, Norway.}
\email{krisshaw@math.uio.no}

\begin{abstract}
We extend the definitions of Chern-Schwartz-MacPherson (CSM) cycles of matroids to tropical manifolds. To do this, we provide an alternate description of CSM cycles of matroids which is invariant under integer affine transformations.
Utilising results of Esterov and Katz-Stapledon, we prove correspondence theorems for the CSM classes of tropicalisations of subvarieties of toric varieties. 
We also provide an adjunction formula relating the CSM cycles of a tropical manifold and a codimension-one tropical submanifold. 
Lastly, we establish Noether's Formula for compact tropical surfaces with a Delzant face structure. This extends the class of surfaces for which the formula had been previously proved by the third author. 
\end{abstract}

\maketitle
\tableofcontents


\section{Introduction}

The Chern-Schwartz-MacPherson (CSM) cycles of a matroid $M$ are a collection of tropical cycles introduced in \cite{LdMRS}. 
When the matroid is realisable by a hyperplane arrangement in $\CC^n$, 
its CSM cycles encode the CSM classes of the complement of the arrangement inside its 
wonderful compactification.

CSM cycles of matroids geometrically encode many properties of the matroid. 
Their degrees, for example, are the coefficients of the shifted reduced characteristic
polynomial; 
this was proved via deletion-contraction in \cite{LdMRS}, and later by \cite{AshrafBackman} using a refined basis activities expansion of the Tutte polynomial. 
The fact that the coefficients of the shifted reduced characteristic polynomial of a matroid have this tropical-geometric interpretation was used in 
\cite{ADH} to prove the log-concavity 
of this sequence.
CSM cycles of matroids have also been related to tautological classes of matroids in \cite{BEST}.

In this paper we extend the definition of CSM cycles of matroids to (smooth) tropical manifolds,
and we give evidence that they behave like Chern classes of tangent bundles.
Tropical manifolds are topological spaces 
equipped with an atlas of charts to Bergman fans of matroids 
and transition maps that are integer affine maps; 
see Definition \ref{def:tropmanifold} for more details. 
Examples of tropical manifolds include tropical curves, tropical linear spaces, 
non-singular tropical hypersurfaces, and integral affine manifolds; see Section \ref{sec:examples}.

The CSM cycles of a matroid $M$ are Minkowski weights supported on the different skeleta of the Bergman fan of $M$. 
The weights assigned to the faces of the skeleta are given by products of beta invariants of matroid minors of $M$; see Definition \ref{def:chernweight}. 
A priori, this means that the CSM cycles of $M$ depend heavily on $M$. 
There are examples of non-isomorphic matroids whose corresponding Bergman fans have the same support up to integer linear maps; see, for example, Proposition \ref{prop:parallelBergman}.  
As a first step in extending the definition of CSM cycles to tropical manifolds, we prove in Proposition  \ref{thm:csmisomorphism} that the definition of CSM cycles of matroids 
is ${\rm GL}_n(\Z)$ invariant. 
For this purpose, we show that the weights of the CSM cycles can be equivalently determined from the cosheaves arising in tropical homology \cite{IKMZ}. 

In Section \ref{sec:correspondence}, we provide results that relate the CSM classes of algebraic varieties with the CSM cycles of their tropicalisations.
We first show that the $0$-th CSM cycle of the tropicalisation of a family of subvarieties of a toric variety recovers the topological Euler characteristic of a general member of the family.

\newtheorem*{thm:Eulercharacteristic}{Theorem \ref{thm:Eulercharacteristic}}
\begin{thm:Eulercharacteristic}
Let $\mathcal{X}$ be a meromorphic family of subvarieties of a toric variety $\mathcal{Y}$ over the punctured disk $\mathcal{D}^*$, with general member $\XX_t$. Suppose that $\trop(\mathcal{X}) = X$ is a tropical submanifold in $Y = \trop(\mathcal{Y})$. Then 
$$\chi(\mathcal{X}_t) = \deg(\csm_0(X)),$$
where $\chi(\mathcal{X}_t)$ denotes the topological Euler characteristic. 
\end{thm:Eulercharacteristic}

To prove the  above theorem, we combine the correspondence  results for CSM cycles of matroid fans from \cite{LdMRS} and 
the description of the tropical motivic nearby fibre of Katz and Stapledon \cite{KatzStapledon}.

In the case $\mathcal Y$ is a non-singular projective toric variety, we relate the CSM cycles of $X = \trop(\mathcal{X})$ to the CSM classes of a general member $\XX_t$ of the family $\XX$.
To do this, we combine the above theorem and the work of Esterov \cite{Esterov}, and relate the CSM class of $\XX_t$ in the Chow ring of $\mathcal{Y}$ to the CSM cycles of the tropicalisation of $\XX$, as we now explain. 

The Chow cohomology ring of a non-singular projective toric variety $\mathcal Y$ defined by a fan $\Sigma$ is isomorphic to the ring of Minkowski weights supported on $\Sigma$ \cite{FultonSturmfels}. We denote this graded ring  by $\MW_{\ast}(\Sigma)$. 
The group of Minkowski weights of dimension $k$ is denoted $\MW_k(\Sigma)$, and is isomorphic to $\Hom ( \A_{n-k}(\mathcal{Y}) , \Z)$, where $n$ is the dimension of $\mathcal Y$ and $\A_{n-k}(\mathcal{Y})$ is the $(n-k)$-graded piece of the Chow group of $\mathcal{Y}$. 
By Poincar\'e duality, taking the cap product with the fundamental class of $\mathcal{Y}$ provides an isomorphism
$$\cdot \frown  [\mathcal{Y}] \colon \MW_{k}(\Sigma) \xrightarrow{\,\,\cong\,\,} \A_{k}(\mathcal{Y}).$$
If $\mathcal{X}$ is transverse to the toric boundary of $\mathcal{Y}$ then the recession fan of $X$ is supported on the fan $\Sigma$. Following \cite{AHR}, by taking the recession cycles $\rec(\csm_k(X))$ of $\csm_k(X)$ (see Definition \ref{def:reccycle}), we thus obtain a collection of Minkowski weights on $\Sigma$.

\newtheorem*{thm:corrToric}{Theorem \ref{thm:corrToric}}
\begin{thm:corrToric}
Suppose $\mathcal{X}$ is a meromorphic family of subvarieties of $(\C^*)^N$ over the punctured disk $\mathcal{D}^*$, with general member $\XX_t$. 
Let $\mathcal{Y}$ be a non-singular projective toric variety 
with defining fan $\Sigma$, and let $\overline{\mathcal{X}}$ be the closure of $\mathcal{X}$ in $\mathcal{Y} \times \mathcal{D}^*$. 
Suppose that $\overline{X} := \trop(\overline{\mathcal{X}})$ is  a tropical manifold transverse to $\trop(\mathcal{Y})$. Then
$$\csm_k( \mathbbm{1}_{\mathcal{X}_t}) \, = \, \rec(\csm_{k}(X)) \frown [\mathcal{Y}] 
\,\, \in \A_{k}(\mathcal{Y}).$$
\end{thm:corrToric}

Theorems \ref{thm:Eulercharacteristic} and \ref{thm:corrToric} generalize \cite[Theorem 3.1]{LdMRS}, which relates the CSM class of the complement of a hyperplane arrangement in the Chow ring of its wonderful compactification 
to the CSM cycles of the corresponding matroid. 

In previous work, 
Bertrand and Bihan equipped the skeleta of tropical complete intersections in $\R^n$ with integer weights 
to produce balanced tropical cycles \cite{BertrandBihan}.  
In their construction, their weights are, up to  sign, related to the Euler characteristic of a non-degenerate complete intersection 
in $(\CC^*)^n$ \cite[Theorem 5.9]{BertrandBihan}. 
The situation they consider overlaps with our own when the tropical complete intersections are also tropical manifolds. In the hypersurface case, this is equivalent to the dual subdivision being unimodular. 

In Section \ref{sec:adjunction}, we give a product formula for CSM cycles of matroids, and use it to prove an adjunction-like formula relating the codimension-1 CSM cycles of a tropical manifold and a submanifold of codimension 1. This generalises the adjunction formula for curves in tropical surfaces proved in \cite[Theorem 6]{Shaw:Surf}.

\newtheorem*{thm:Adjunction}{Theorem \ref{thm:Adjunction}}
\begin{thm:Adjunction}[Adjunction formula]
Let $X$ be a tropical manifold  of dimension $d$ and $D \subset X$ a tropical submanifold of codimension 1 in $X$. 
Then 
$$\csm_{d-2}(D) = (\csm_{d-1}(X) -  D )\cdot D,$$
where $\cdot$ denotes the tropical intersection product in $X$. 
\end{thm:Adjunction}

The notion of a tropical submanifold is more subtle than just being a tropical manifold contained in a tropical manifold; see Definition \ref{def:submanifold}. In Example \ref{ex:singline} we recall a situation where being a tropical manifold of contained in $X$ and having codimension 1 is not sufficient for Theorem \ref{thm:Adjunction}.

In the case $X$ is a 2-dimensional compact tropical manifold, we prove a tropical version of Noether's Formula expressing its topological Euler characteristic $\chi(X)$ in terms of its CSM cycles and their intersections. This is proven under the assumption that $X$ admits a Delzant face structure, i.e., a cell structure where every cell locally looks like a polytope in which the primitive vectors of the edges adjacent to any vertex can be extended to a basis of the ambient lattice; see Definition \ref{def:facestructure}.
 
\newtheorem*{thm:Noether}{Theorem \ref{thm:Noether}}
\begin{thm:Noether}[Tropical Noether's Formula]
Let $X$ be a compact tropical surface admitting a Delzant face structure.
Then 
$$\chi(X) = \frac{\deg( \csm_0(X) + \csm_1(X)^2) }{12}.$$
\end{thm:Noether}
Noether's formula had previously been proven for tropical surfaces arising from tropical toric surfaces via the operations of tropical modifications and summation \cite{Shaw:Surf}.

\subsection*{Acknowledgements}
We would like to give our thanks to Alex Esterov, Kyle Huang, Grigory Mikhalkin, Dmitry Mineyev, Johannes Rau, and Paco Santos for very helpful conversations.  

L.~L.d.M.~was funded by ECOS NORD 298995 and PAPIIT-IN105123. 
F.~R.~was partially supported by EPSRC grant EP/T031042/1.
K.~S.~is supported in part by the Trond Mohn Foundation project ``Algebraic and topological cycles in complex and tropical geometry". 
This collaboration was also carried out under the  Center for Advanced Study Young Fellows Project ``Real Structures in Discrete, Algebraic, Symplectic, and Tropical Geometries", as well as during the workshops Tropical Methods in Geometry in the MFO in Oberwolfach and Algebraic Aspects of Matroid Theory in BIRS in Banff. We thank all  institutions for the their support and excellent working conditions.

\section{Preliminaries}
 
\subsection{Tropical manifolds}
To define tropical manifolds we first introduce a partial compactification of $\R^n$ used in tropical geometry. Let $\mathbb{T} := \R \cup \{-\infty\}$ and equip it with the topology whose basis consists of intervals $(a, b)$ and $[-\infty, b)$ where $a < b$. The real numbers $\R$ will be equipped with the Euclidean topology. The spaces $\T^r$ and $\T^r \times \R^n$ are equipped with the product topologies. All subsets of $\T^r \times \R^n$ are equipped with the subspace topologies. 

The space $\T^r$ is a stratified,  where the strata are 
$$\T^r_I := \{ (x_1, \dots , x_r)  \colon x_i = -\infty \text{ iff } i \in I\}$$ for 
$I \subset [n]$.
Given a stratum $\T^r_I$, we call $I$ its \textbf{sedentarity}, and we say  a point $x \in \T^r$ is of  sedentarity $I$ if and only if it is in $\T^r_I$.  Notice that we have 
$\T^r = \bigsqcup_{I \subset [r]} \T^r_I$ 
and $\T^r_I$ can be identified with $\R^{r -|I|}$.   

We assume the reader is familiar with the definitions of tropical cycles in $\R^n$, including weight functions and the balancing condition, and refer to \cite{MaclaganSturmfels} and \cite{MikRau} for more background. 
A \textbf{tropical cycle of sedentarity $I$} in $\T^r$ is the topological closure of a tropical cycle in $\T^r_I \cong \R^{r-|I|}$. 
Notice that we do not ask for a cycle in $\T^r$ of sedentarity $I$ to satisfy a balancing condition on any new codimension-1 faces in the closure.
 A tropical cycle in $\T^r$ is a formal sum of tropical cycles of different  sedentarities. 

We can also extend the notion of  sedentarity  to spaces of the form $\mathbb{T}^{r} \times \R^n$. A tropical cycle of sedentarity $I$ is the closure  in $\mathbb{T}^{r} \times \R^n$ of a cycle in $\mathbb{T}^r_I \times \R^n$.   A $k$-dimensional  tropical cycle in $\mathbb{T}^r_I$ is again a formal sum of $k$-dimensional cycles of different sedentarities.

\begin{exa}[{\em Matroidal tropical cycles.}]
Denote by $\{e_0, e_1, \dotsc, e_n\}$ the standard basis of the lattice $\ZZ^{n+1} \subseteq \R^{n+1}$. 
For any subset $S \subseteq \{0,\dotsc,n\}$, let $e_S \coloneqq \sum_{i\in S} e_i \in \ZZ^{n+1}$.
We denote $\mathbf 1 := (1,1,\dots,1) \in \RR^{n+1}$.

Suppose $M$ is a loopless matroid of rank $d+1$ on the ground set $\{0,1,\dots,n\}$. 
The {affine Bergman fan} $\hat{\B}(M)$ of $M$ 
is the pure $(d+1)$-dimensional rational polyhedral fan in $\R^{n+1}$ consisting of the collection of cones of the form 
\[\sigma_\F \coloneqq  -\cone(e_{F_1},e_{F_2},\dotsc,e_{F_k}) + \RR \!\cdot\! \mathbf 1\]
where 
$\F = \{\emptyset \subsetneq F_1 \subsetneq F_2 \subsetneq \dotsb \subsetneq F_k \subsetneq \{0,\dotsc, n\}\}$ is a chain of flats in the lattice of flats $\L(M)$ of $M$. 
If $M$ is a matroid with loops then we define $\hat{\B}(M) = \emptyset$.
The {(projective) Bergman fan} $\B(M)$ of $M$
is the pure $d$-dimensional rational polyhedral fan obtained as the image of $\hat{\B}(M)$ 
in the quotient vector space $\R^{n+1} / \,\RR \!\cdot\! \mathbf 1$.

The {\bf matroidal tropical cycle} $Z_M$ associated to $M$ is the tropical cycle in the vector space $\R^{n+1} / \,\RR \!\cdot\! \mathbf 1 \cong \R^n$ whose support is the projective Bergman fan 
$\B(M)$ and all multiplicities on the maximal cones are equal to $1$. 
Note that the product $ X = \T^r \times Z_M$ is a tropical cycle in $\mathbb{T}^r \times \R^n$ of sedentarity $\emptyset$ --- indeed, $X$ is the closure of the matroidal tropical cycle $Z_{C_r \oplus M} \subseteq \R^r \times \R^n$, where $C_r$ is the free matroid on $r$ elements (i.e., having no circuits).
\end{exa}

Spaces of the form $\T^r \times Z$ for $Z$ a matroidal cycle in $\R^n$ will be the local building blocks of tropical manifolds. Before we give the precise definition we must also define the notion of {extended integer affine maps}. Recall that an integer affine map $F: \R^{r'} \to \R^r$ is the composition of an integer linear map with a translation defined over $\R$, that is, a map of the form $F(x) = Ax + b$ with $A \in \ZZ^{r \times r'}$ and $b \in \RR^r$.

\begin{defi}
\label{def:integeraffmap}
Let $F\colon  \R^{r'} \rightarrow \R^r$ be an integer affine map and let $A \in \ZZ^{r \times r'}$ denote the integer matrix representing the linear part of $F$. Let $I$ be the set of 
$i \in [r']$ such that the $i$-th column of $A$ has only non-negative
entries. 
Then $F$ can be extended to a map 
$$\hat F\colon \left( \bigcup \limits_{J \subset I} \T^{r'}_J \right)\rightarrow \TT^r$$ 
by continuity. 
The restriction of such an $\hat F$ to an open subset $U' \subset \TT^{r'}$ is called an \textbf{extended integer affine map}. Note that this only makes sense if we have $\text{sed}(x) \subset I$ for all $x \in U'$. 
\end{defi}

We are now ready to define tropical manifolds. These will be topological spaces with an atlas of charts to the supports of matroid fans and extended integer affine linear  transition maps. 
We follow the definitions from  \cite[Section 7]{MikRau} for abstract tropical cycles and polyhedral spaces,
and collect the necessary properties from \cite[Definition 7.4.1]{MikRau} in the definition below. 

\begin{defi} \label{def:tropmanifold}
A $d$-dimensional \textbf{tropical manifold} is a Hausdorff topological space $X$ equipped with an atlas of charts  $\{\varphi_{\alpha} \colon U_{\alpha} \rightarrow \Omega_{\alpha} \subset X_{\alpha}\}_{{\alpha} \in \mathcal I}$ where $\mathcal I$ is a finite set 
and 
\begin{enumerate}
\item $X_{\alpha} = \mathbb{T}^{r_{\alpha}} \times Y_{\alpha}$ with $Y_{\alpha}$ a matroidal tropical cycle of dimension $d- r_{\alpha}$ in $\R^{n_{\alpha}}$ and $\Omega_\alpha$ is  an open subset of $X_{\alpha}$ for every $\alpha \in \mathcal I$, 
\item for every $\alpha \in \mathcal I$ the map $\varphi_{\alpha} \colon U_{\alpha} \rightarrow \Omega_{\alpha}$ is a homeomorphism, 
\item the transition functions $\varphi_{\beta} \circ \varphi_{\alpha}^{-1}$ are extended integer affine maps,
\item \label{closure} for each $\alpha \in \mathcal I$ there exists an extension $\varphi_{\alpha}' \colon U_{\alpha}' \rightarrow \Omega_{\alpha}' \subset X_{\alpha}$ of $\varphi_\alpha$ such that $\overline{\Omega}_{\alpha} \subset  \Omega_{\alpha}' $. 
\end{enumerate}
\end{defi}

\subsection{The boundary of a tropical manifold}
\label{sec:boundary}

Points of the stratified space $\T^r \times \R^n$ have an {order of sedentarity} defined by
$$s(x) := |\{ x_i : x_i = -\infty\}|.$$
The order of sedentarity is preserved under invertible extended integer affine maps and hence this notion extends to tropical manifolds, as long as we assume that the chart chosen contains points of sedentarity $0$ in their image. 

\begin{defi}
For a point $x \in  X$ of a tropical manifold, the {\bf order of sedentarity} of $x$ is $s(x) = |s(\phi_{\alpha}(x))|$, where $\phi_{\alpha}: U_{\alpha} \to  X_{\alpha}$ is any  chart satisfying $x \in U_{\alpha}$ and with $\phi_{\alpha}(U_{\alpha})$ containing points of sedentarity $0$.
The {\bf boundary} of a tropical manifold $X$ is $$\partial X = \{ x \in X : s(x) > 0 \}.$$ 
A {\bf boundary divisor} of a tropical manifold $X$ is 
the closure of a connected component of the set $\{ x \in X : s(x) = 1\}$. 
Every boundary divisor is of codimension 1 in $X$ and is itself a tropical manifold; \cite[Proposition 1.2.8]{Shaw:Surf}. Call the set of boundary divisors $\A = \{ D_1, \dots, D_k\}$ the \textbf{arrangement of boundary divisors} of $X$. 
\end{defi}

For a tropical manifold $X$ we denote the complement of its boundary by $X_0 = X \backslash \partial X$. This is the collection of points of $X$ with order of sedentarity  $0$.
Note that $X_0$ is also a tropical manifold and that $X$ is the closure of $X_0$. 
Each connected component of $\{ x \in X : s(x) = k\}$ is a tropical manifold of dimension $\dim X - k$, and so is its closure in $X$.

\subsection{Examples of tropical manifolds}
\label{sec:examples}

\begin{exa}[{\em Abstract tropical curves.}] An abstract tropical curve is a  $1$-dimensional finite simplicial complex equipped with a metric on the complement of its 1-valent vertices. Abstract tropical curves are exactly $1$-dimensional tropical manifolds, since in dimension $1$ the metric is equivalent to an integer affine structure. Indeed, each point of a tropical curve comes equipped with a chart to a matroidal fan in $\R^n$ of a loopless matroid of rank $2$ or to a neighbourhood of $-\infty \in \T$.
\end{exa}

\begin{exa}[{\em Integral affine manifolds.}]
An integral affine manifold is a manifold equipped with an atlas of charts such that the transition functions are affine transformations with linear part defined by an integral matrix. 
An integer affine manifold is thus a tropical manifold where the tropical charts from Definition \ref{def:tropmanifold} satisfy $X_{\alpha} = \R^d$ for all $\alpha$. In dimension 2 the topological  type of a compact integer affine manifold without boundary is either $S^1 \times S^1$ or the Klein bottle. 

Tropical Abelian varieties are integer affine manifolds obtained as quotients  $\R^n / \Lambda$, where $\Lambda \subseteq \R^n$ is a full-rank sublattice satisfying certain conditions \cite{MikZha:Jac}.  An $n$-dimensional tropical Abelian variety is homeomorphic to $(S^1)^n$. 
\end{exa}

\begin{exa}[{\em Tropical toric varieties.}]\label{ex:toricman}
Let $\Sigma$ be a rational polyhedral fan in $\R^N$.
As a topological space, the associated tropical toric variety $\T\Sigma$ can be described as a quotient of the disjoint union of tropical tori $U_{\sigma} = \T^{\dim(\sigma)} \times \R^{N - \dim(\sigma)},$ ranging over all cones $\sigma \in \Sigma$ \cite[Definition 3.2.3]{MikRau}. 
From this description we obtain an atlas of charts $\{U_{\alpha}, \phi_{\alpha}, X_{\alpha} \}_{\alpha \in I}$ where $I$ is in bijection with the top dimensional cones of $\Sigma$ and the charts $\phi_{\alpha}: U_{\alpha} \to \Omega_{\alpha} \subset X_{\alpha}$ satisfy $X_{\alpha} = \T^N$ for all $\alpha \in I$. 
The tropical toric variety $\T\Sigma$ is a tropical manifold if and only if  $\Sigma \subset \R^N$ is a  unimodular fan, see \cite[Section 3.2]{MikRau}.  The tropical toric variety $\T\Sigma$ is compact if and only if $\Sigma$ is complete.   
\end{exa}

\begin{exa}[{\em Tropical projective space.}]
Projective space is a toric variety and 
tropical projective space can be constructed using the same fan as projective space over a field. It can also be described as the following quotient:
$$\TP^{n} = \frac{\R^{n+1} \backslash ( -\infty, \dots, -\infty)}{(x_0: \dots :x_n) \sim (a + x_0: \dots :a +x_n)  },$$
where $a \in \T \backslash \{-\infty\}$.
From the quotient description we obtain tropical homogeneous coordinates on $\TP^n$ which we write as $[x_0 : \dots : x_n]$.   
\end{exa}

\begin{exa}\label{ex:submanifolds}
[\emph{Tropical manifolds in  tropical toric varieties}]
Let $\hat{X}$ be a pure dimensional polyhedral complex in $\R^n$. The support $X = |\hat{X}|$ is a tropical manifold if for every face $\tau \in \hat X$, the star fan ${\rm star}_\tau(\Sigma)$ of $X$ at  $\tau$ is the support of a Bergman fan of some matroid $M_{\tau}$ up to an integer affine transformation of $\R^n$. 

Suppose in addition that the collection of recession cones of all the cones $\tau \in \hat X$ forms a fan $\Sigma$, called the recession fan of $\hat X$.
Let $\overline{X}$ denote the closure of $X$ in the tropical toric variety $\T \Sigma$.
Then $\overline{X}$ is a compact space and is the canonical compactification of $X$ from \cite{KastnerShawWinz}, \cite{AminiPiquerez}.  Moreover, the space $\overline{X}$ is a tropical manifold in the sense of Definition \ref{def:tropmanifold}. 
\end{exa}

\begin{exa}[{\em Tropical linear spaces.}]
A {tropical Pl\"ucker vector} is a vector $p \in \Rbar^{\binom{[n+1]}{d+1}}$
satisfying the tropical Pl\"ucker relations: 
For every $A \in \binom{[n+1]}{d+2}$ and $B \in \binom{[n+1]}{d}$, the maximum $\max_{i \in A \setminus B} (p_{A \setminus i} + p_{B \cup i})$ is attained twice. 
Any tropical Pl\"ucker vector $p$ gives rise to a {tropical linear space} $\textstyle L(p)$ defined as
\[
\textstyle L(p) := \{x \in \R^{n+1} / \,\RR \!\cdot\! \mathbf 1 : \max_{i \in S} (p_{S-i} + x_i) \text{ is achieved twice for any } S \in \binom{[n+1]}{d+2}\}.
\]
Tropical linear spaces are $d$-dimensional tropical manifolds: around any $x \in \textstyle L(p)$,
the tropical linear space $\textstyle L(p)$ looks like a matroidal fan \cite{Speyer}.

We can consider the closure of a tropical linear space in different tropical toric varieties. For example, the closure of $L(p)$ in $ \TP^n$ is a tropical projective subspace; however, the closure $\overline{L(p)}$ does not necessarily intersect the boundary of $\TP^n$ transversely, and $\overline{L(p)}$ will not be a tropical manifold in the sense of Definition \ref{def:tropmanifold}.
Nonetheless, the closure of $L(p)$ in the tropical toric variety of $\Sigma = \rec(L)$ where $\rec(L)$ denotes the recession fan of $L$ is a tropical manifold in the sense of Definition \ref{def:tropmanifold}.
\end{exa}

\begin{exa}[{\em Non-singular tropical hypersurfaces.}]\label{ex:hyper}
A tropical hypersurface $X_f \subset \R^N$ is the divisor $\text{div}_{\R^N}(f)$ of a tropical regular  function  $f:\R^N  \to \R$, see \cite{MaclaganSturmfels}, \cite{MikICM}. It is a weighted  polyhedral complex dual to a regular subdivision of the Newton polytope of $f$.  If the dual subdivision is unimodular, meaning each polytope in the subdivision is a simplex with normalised volume equal to $1$, the hypersurface is called non-singular. Non-singular hypersurfaces locally look like matroidal tropical cycles of corank-$1$ matroids, and are thus tropical manifolds.

Let $\Sigma$ be the dual fan of the Newton polytope of $f$. 
The recession fan of $X_f$ is the codimension-1 skeleton of $\Sigma$. As in Example \ref{ex:submanifolds}, we can compactify $X_f$ in $\T \Sigma$. If the  dual fan $\Sigma$ is unimodular then $\T \Sigma$ and $\overline{X_f}$ are both compact tropical manifolds in the sense of Definition \ref{def:tropmanifold}.
\end{exa}

\subsection{Cycles in tropical manifolds}
Here we recall the definitions of cycles in tropical manifolds from \cite{ShawInt} and  \cite{MikRau}. 
As previously stated, a tropical cycle $A$ in $\R^n \times \T^r$ is a formal sum of cycles $A = \sum_{I \subset [r]} A_I$ where $A_I$ is the closure in $\R^n \times \T^r$ of a tropical cycle in 
$\R^n \times \TT^r_I = \R^n \times \RR^{r - |I|}$.  

\begin{defi}\label{defi:tropicalCycles}
A \textbf{tropical $k$-cycle} $Z$ in a tropical manifold $X$ is a subset $Z \subset X$ equipped with a weight function $w: \Omega \to \Z$ on an open dense subset $\Omega \subset Z$ such that, for all charts $\varphi_{\alpha} \colon U_{\alpha} \to Y_{\alpha} \times \T^{r_{\alpha}}$ of $X$, the image $\varphi_{\alpha}(Z \cap U_{\alpha})$ is a tropical $k$-cycle in $\R^{n_{\alpha}} \times \T^{r_{\alpha}}$ with the weight function induced from the weight function on $Z \cap U_{\alpha}$. 
\end{defi}

\begin{defi}\label{def:transverse}
A tropical $k$-cycle $A$ of a tropical manifold $X$ is {\bf transverse to the boundary} of $X$ if for every boundary stratum  $D_{I} := \bigcap_{i \in I} D_i$ of $X$ and for every chart $\varphi_{\alpha} \colon U_{\alpha} \to Y_{\alpha} \times \T^{r_{\alpha}}$, the image $\varphi_{\alpha}(U_\alpha \cap A \cap D_I)$ is of codimension $|I|$ in $\varphi_{\alpha}(U_\alpha \cap A)$  or it is empty.
\end{defi}

The adjunction formula for tropical CSM cycles in Section  \ref{sec:adjunction} concerns the intersection of divisors in a tropical manifold. 
Similarly, Noether's Formula in Section \ref{sec:Noether} talks about the 
intersection of $1$-cycles in tropical surfaces. 
There are various approaches to intersection theory in tropical manifolds \cite{ShawThesis}, \cite{FrancoisRau}, all of them equivalent for our purposes. 
For completeness, we now recall the theory of intersecting with tropical Cartier divisors. 

A {\bf tropical Cartier divisor} on a tropical manifold $X$ equipped with charts $\varphi_{\alpha} \colon U_{\alpha} \to X_{\alpha} \subset \RR^{n_\alpha} \times \T^{r_{\alpha}}$ is a collection of tropical rational functions $\{f_{\alpha}\}$ where 
$f_{\alpha} \colon \R^{n_{\alpha}} \times \T^{r_{\alpha}} \to \TP ^1 = [-\infty, \infty]$ and such that 
on the overlaps $U_{\alpha} \cap U_{\beta} $ the difference $f_{\alpha} - f_{\beta} $ is a 
bounded integer affine function on $X_{\alpha}$, implying in particular that $f_{\alpha} - f_{\beta} $ does not attain the values $\pm \infty$.
Every codimension-$1$ tropical cycle $D$ in a tropical manifold is a Cartier divisor \cite[Lemma 2.23]{ShawInt}, 
meaning that 
there exists a tropical Cartier divisor $\{f_{\alpha}\}$ such that $\varphi_\alpha (D \cap U_{\alpha}) = \Div_{X_{\alpha}}(f_{\alpha})$ for all $\alpha$.  

Suppose  $X$ is a tropical manifold without boundary. 
Then for $D$ a codimension-$1$ tropical cycle in $X$ and $A$ a tropical $k$-cycle in $X$, we can define the intersection of $D$ and $A$ by first expressing $D$ as a Cartier divisor $f = \{ f_{\alpha}\}$, and then setting
$$D \cdot A = \Div_A( f ),$$
where $\Div_A(f)$ is the tropical cycle which in each chart $U_{\alpha}$ is equal to $\Div_{A \cap U_{\alpha}}(f_{\alpha})$. 
The above recipe also works in the case $X$ has a boundary and $D$ does not contain any boundary divisors.

When the codimension-$1$ cycle $D$ contains components which are boundary divisors it is possible that some of the functions $\{f_{\alpha}\}$ are identically equal to $\pm \infty$  on $A$. This is the case, for example, when we wish to consider $D^2$ for a boundary divisor $D$. In this situation we use the theory of tropical line bundles from \cite{JRS}. The charts of $X$ restricted to a neighbourhood of $D$ in $X$ defines a line bundle on $D$ \cite{ShawThesis}. 
The Cartier divisor $\{f_{\alpha}\}$ defines a tropical line bundle $L \in \Pic(X) = H^1(X, \Aff_{\mathbb{Z}})$ and 
all Cartier divisors rationally equivalent to $f$ arise from tropical rational sections of $L$. By \cite[Proposition 4.6]{JRS}, every tropical line bundle admits a non-zero section $s$. A choice of section produces another codimension-$1$ cycle $D'$ in $X$ which is rationally equivalent to $D$ and is a sedentarity $0$ cycle. This cycle can then be intersected with $A$, so we define $D \cdot A := D' \cdot A$.
Notice that in this case the intersection product is only well defined up to rational equivalence, since $D'$ depends on the choice of section of $L$.

\subsection{Tropical submanifolds}
\label{sec:submanifolds}
For a tropical subvariety to behave as a submanifold of a tropical manifold we require a compatibility condition on the corresponding charts, as described below.

\begin{defi}\label{def:submanifold}
Let $X$ be a tropical manifold and let $W$ be a tropical cycle in $X$ of dimension $k$ which has constant weight function equal to $1$. 
Then $W$ is a {\bf tropical submanifold} of $X$ if 
there is an atlas of charts $\{U_{\alpha}, \varphi_{\alpha}\}$ for $X$ 
such that $\{(U_{\alpha} \cap W), \varphi_{\alpha}|_W\}$ is also an atlas of charts for $W$. 
\end{defi}

Note that, in particular, if the atlas for $X$ has charts $\varphi_{\alpha} \colon U_{\alpha} \to \Omega_{\alpha} \subseteq Y_{\alpha} \times \T^{r_\alpha}$ where $Y_{\alpha}$ are matroid fans, then $\varphi_{\alpha}|_W \colon U_{\alpha} \cap W \to W_{\alpha} \times \T^{r'_\alpha}$, where $W_{\alpha}$ are matroid fans such that $W_{\alpha} \subseteq Y_{\alpha}$  and $r'_\alpha \leq r_\alpha$.  This implies that the matroid of $W_{\alpha}$ must be a matroid quotient of the matroid of $Y_{\alpha}$  \cite[Lemma 2.21]{ShawInt}, \cite[Proposition 3.3]{FrancoisRau}. Example \ref{ex:singline} shows a pair of tropical manifolds $W$ and $X$ with $W \subset X$ which violates this condition on the compatibility of the charts and hence $W$ is not a submanifold of $X$ in the sense of Definition \ref{def:submanifold}.

We say that $W$ is a sedentarity-$0$ submanifold of $X$ if $r'_\alpha=r_\alpha$ in all charts $\varphi_\alpha$ of $X$. By the condition imposed on the compatibility of the charts, a sedentarity-$0$ submanifold of $X$  is necessarily transverse to the boundary of $X$ in the sense of Definition \ref{def:transverse} when considered as a tropical cycle of $X$. 

\begin{exa*}\label{ex:singline}
As we have mentioned before, being a tropical submanifold depends on a compatibility condition on the charts. 
For instance, there exist tropical surfaces $X$ in $\R^3$ of degree $\geq 3$ which contain tropical lines $L$ not satisfying Definition \ref{def:submanifold}.
These examples were discovered by Vigeland \cite{Vig1} as infinite families of lines on tropical hypersurfaces.
Both the surface $X$ and the line $L$ are tropical manifolds;
however, the atlases for $X$ and $L$ are not compatible with each other. Moreover, these lines in tropical surfaces do not satisfy the adjunction-like formula that we present in Theorem \ref{thm:Adjunction};  see \cite[Section 7.3]{BrugalleShaw}. 
\end{exa*}

\setcounter{thm}{0}

\section{CSM cycles of tropical manifolds}

In this section we show that CSM cycles of matroids are invariant under invertible integer affine transformations of the underlying Bergman fan, which allows us to extend the definition of CSM cycles to tropical manifolds; see Definition \ref{def:csmManifold}. 
We start by recalling the definition from \cite{LdMRS}.

\begin{defi}\label{def:chernweight}
Suppose $M$ is a loopless rank $d+1$ matroid on the ground set $\{0,1,\dots,n\}$.
For $0 \leq k \leq d$, the $k$-dimensional {\bf Chern-Schwartz-MacPherson (CSM) cycle}
of $M$, denoted $\csm_k(M)$, is the tropical cycle in $\R^{n+1} / \, \mathbb R \cdot \mathbf{1} \cong \RR^n$ supported on the $k$-dimensional skeleton 
of the Bergman fan $\B(M)$ in which the weight of the top-dimensional cone $\sigma_\F$ corresponding to a flag of flats $\F \coloneqq  \{\emptyset = F_0
\subsetneq F_1 \subsetneq \dotsb \subsetneq F_{k} \subsetneq F_{k+1} =
\{0,\dotsc,n\}\}$ is 
\[ w_{\csm_k(M)}(\sigma_\F) \coloneqq  (-1)^{d-k} \prod_{i=0}^{k} \beta(M|F_{i+1}/F_i),\]
where $M|F_{i+1}/F_i$ denotes the minor of $M$ obtained by restricting to $F_{i+1}$ and contracting $F_i$.
If $M$ is a matroid with loops then we define $\csm_k(M) \coloneqq  \emptyset$ for all $k$. If $\Sigma_M$ is the matroidal fan of $M$, we will also denote $\csm_k(\Sigma_M) := \csm_k(M)$.
\end{defi}

The next definition extends the notion of CSM cycles to spaces of the form $X =  \T^r \times \Sigma_M$.
The space $X$ is the closure of the matroidal fan of the matroid $C_r \oplus M$, where $C_r$ is the  free matroid on $r$ elements (i.e., the matroid with no circuits).
Note that for any $I \subset [r]$, the intersection $X \cap (\T^r_I \times \R^n)$ is again the support of a matroid fan in  $\R^{r+n -|I|}$ where the underlying  matroid is $C_{r-|I|} \oplus M$. 

\begin{defi}\label{def:csmInTn}
Let $\Sigma_M$ be a matroid fan in $\R^n$.  The $k$-th Chern-Schwartz-MacPherson cycle of 
$\T^r \times \Sigma_M$  is  
$$\csm_k( \T^r \times \Sigma_M )  = \sum_{\emptyset \subset I \subset [r]} \csm_k(\T^r_I \times \Sigma_M ).$$
\end{defi}

The following 
proposition, whose proof we provide later in this subsection, shows that our definition of CSM cycles is well-behaved under extended integer affine  maps. This allows us to define CSM cycles of general tropical manifolds in Definition \ref{def:csmManifold}. 
 
\begin{prop}\label{thm:csmisomorphism}
Let $Y'$ and $Y$ be matroidal cycles in $\R^{n'}$ and $\R^n$ respectively, and suppose that there is an invertible map $\varphi: Y' \to Y$ which is induced by  an affine  linear map $\varphi \colon \R^{n'} \to \R^{n}$.
Then $$\varphi( \csm_k(Y')) =  \csm_k(Y)$$ as tropical cycles. 
\end{prop}

Since tropical manifolds are defined as spaces with an atlas of charts in which the transition functions are invertible integer affine maps, Proposition \ref{thm:csmisomorphism} ensures the following notion of CSM cycles for tropical manifolds is well defined.

\begin{defi}\label{def:csmManifold}
The $k$-th \textbf{Chern-Schwartz-MacPherson cycle} $\csm_k(X)$ of a tropical manifold $X$ is the tropical cycle supported on $X$ such that, in each chart $\varphi_{\alpha} \colon U_{\alpha} \to X_{\alpha}$, the image of $\csm_k(X)$ is $\csm_k(X_{\alpha}) \cap \varphi_{\alpha} (U_{\alpha})$. 
\end{defi}

\begin{rem}
We remark that the CSM cycles of products $\Sigma_M \times \T^{r}$, where $\Sigma_{M} \subset \R^n$ is a matroid fan, are transverse to the boundary of $\R^{n} \times \T^{r}$, and hence CSM cycles of a tropical manifold $X$ are transverse to its boundary. 
\end{rem}

Before proving Proposition \ref{thm:csmisomorphism}, we provide an interesting class of examples of non-isomorphic matroids whose matroidal cycles are related by invertible integer linear maps.

\begin{defi}\label{defi:parallel}
Let $M_1$ and $M_2$ be two loopless matroids on the ground sets $E_1$ and $E_2$, respectively, and let $p_1 \in E_1$ and $p_2 \in E_2$. The {\bf parallel connection} $P(M_1,M_2)$ of $M_1$ and $M_2$ at the basepoints $p_1$ and $p_2$ is a matroid on the ground set $E := (E_1 - p_1) \sqcup (E_2 - p_2) \sqcup \{p\}$. If we make the identification $p_1 = p_2 = p$, the flats of $P(M_1,M_2)$ are those $F \subseteq E$ such that $F \cap E_1$ is a flat of $M_1$ and $F \cap E_2$ is a flat of $M_2$. Equivalently, the circuits of $P(M_1,M_2)$ are the subsets of $C \subseteq E$ such that $C$ is a circuit of $M_1$ or $M_2$, or $C = I_1 \sqcup I_2$ with $I_i \sqcup \{p\}$ a circuit of $M_i$; see \cite[Proposition 7.6.6]{White1}.
\end{defi}

In general, the isomorphism class of the parallel connection of two matroids $M_1$ and $M_2$ depends on the choice of basepoints. It follows from the next proposition that the matroidal cycles of any two parallel connections of $M_1$ and $M_2$ are related by an invertible integer affine map, even though their underlying matroids might not be isomorphic.

\begin{prop}\label{prop:parallelBergman}
If $M$ is the parallel connection of the matroids $M_1$ and $M_2$ (at any basepoints) then there exists an invertible map
$\varphi: Z_M \to Z_{M_1} \times Z_{M_2}$, which is the restriction of an invertible integer linear map between the ambient vector spaces.  
\end{prop}
\begin{proof}
Let $E_1$ and $E_2$ be the ground sets of $M_1$ and $M_2$, respectively. Suppose $M$ is the parallel connection of $M_1$ and $M_2$ at the base points $p_1 \in E_1$ and $p_2 \in E_2$, and let $E := (E_1 - p_1) \sqcup (E_2 - p_2) \sqcup \{p\}$ be the ground set of $M$. Consider the vector spaces $V_1 := \RR^{E_1}$, $V_2 := \RR^{E_2}$, and $V := \RR^{E}$, and let $\Delta_1 = (e_a)_{a \in E_1}$, $\Delta_2 = (e_b)_{b \in E_2}$, and $\Delta = (e_c)_{c \in E}$ denote their standard bases, respectively. Let $\varphi: V/ \RR \cdot e_{E} \to (V_1/ \RR \cdot e_{E_1}) \oplus (V_2/ \RR \cdot e_{E_2})$ be the map defined by $\varphi(\bar e_x) := \bar e_x$ if $x \in E - p$ and $\varphi(\bar e_p) = \bar e_{p_1} + \bar e_{p_2}$. The function $\varphi$ is an isomorphism of vector spaces; in fact, its inverse is given by $\varphi^{-1}(\bar e_x) = \bar e_x$ if $x \in (E_1 - p_1) \cup (E_2 - p_2)$, and $\varphi^{-1}(\bar e_{p_i}) = -\bar e_{E_i-p_i}$ for $i=1,2$. 

The support of the matroidal cycle
$Z_M$ of any matroid $M$  can be described as
\[\textstyle \bigl|Z_M\bigr| = \bigl\{\sum \alpha_i \bfe_i : 
\forall \text{ circuits $C$ of $M$, } \min \{\alpha_i : i \in C\} \text{ is achieved at least twice} \bigr\},\] see \cite[Theorem 4.2.6]{MaclaganSturmfels}. 
Using the description of the circuits of $M$ from Definition \ref{defi:parallel}, it is easy to check that for any $x \in V/ \RR \cdot e_{E}$ we have $x \in |Z_M|$ if and only if $\varphi(x) \in |Z_{M_1} \times Z_{M_2}|$, as desired.
\end{proof}

\begin{conj}\label{conj:iso}
If $M$ and $M'$ are non-isomorphic matroids and there exists an invertible map  $\phi: Z_M \to Z_{M'}$  which is the restriction of an invertible integer linear map between their ambient spaces, then $M$ and $M'$ are parallel connections.   
\end{conj}

We will prove Proposition \ref{thm:csmisomorphism} by 
giving an alternate description of the weights of the CSM cycles; see Lemma \ref{lem:sameweights}. 
It will be clear that this description is invariant under invertible integer affine maps. 
For this purpose, we make use of the sheaves arising from tropical homology \cite{IKMZ}.

\begin{defi}\label{def:Fp}
Let $\f$ be a pure $d$-dimensional polyhedral fan in an $\RR$-vector space $V$. 
For $0\leq p\leq d$, define 
\[\Fs_p(\f) := \langle v_1 \wedge \dots \wedge v_p : v_1, \dotsc , v_p \text{ are in a common cone } \sigma \in \f \rangle \, \subseteq  \, \bigwedge\nolimits^p V.\]
By convention, if $\f \neq \emptyset$ then $F_0(\f)$ is the one-dimensional vector space $\bigwedge\nolimits^0 V$, and if $\f = \emptyset$ then $F_0(\f) = 0$. 
\end{defi}

\begin{lemma}\label{lem:invarianceFp}
Let $\f$ be a polyhedral fan in an $\RR$-vector space $V$. Then the vector spaces $\Fs_p(\f)$ from Definition \ref{def:Fp} are invariants of the support of $\Sigma$. 
Moreover, if $\f'$ is a polyhedral fan in an $\RR$-vector space $V'$ and $\phi: \f \to \f'$
is an invertible map induced by a linear map $\hat{\phi}: V \to V'$ then $\dim \Fs_p(\f) = \dim \Fs_p(\f')$ for all $p$. 
\end{lemma}

\begin{proof}
To prove that $\Fs_p(\f)$ is an invariant of the support of $\f$, it suffices to show that $\Fs_p(\f) = \Fs_p(\f')$ for $\f'$ a refinement of $\f$. Given a cone $\sigma \in \f$, let $L(\sigma)$ denote its linear span. Note that
$$\Fs_p(\f) = \big \langle \bigwedge  ^p L(\sigma) : \sigma \in \f \rangle.$$
By definition of a refinement, every cone of $\f^{'}$ is contained in a cone of $\f$ and thus we have an immediate inclusion $\Fs_p(\f') \subset \Fs_p(\f)$. 
On the other hand, given any cone $\sigma \in \f$ there must exist a cone $\sigma' \in \f^{'}$ contained in $\sigma$ such that $\dim(\sigma') = \dim(\sigma)$ and thus $L(\sigma') = L(\sigma)$.
This provides the inclusion $\Fs_p(\f) \subset \Fs_p(\f')$. 

For the second part of the statement, 
if $\f$ and $\f'$ are linearly isomorphic as described, the isomorphism $\hat{\phi}: V \to V'$ induces an isomorphism between $\Fs_p(\f) $ and $\Fs_p(\f')$. In particular, these vector spaces have the same dimension for each $p$, as claimed. 
\end{proof}

\begin{prop}\label{prop:charOS}
If $\f = \B(M)$
is the Bergman fan of a loopless rank-$(d+1)$ matroid $M$ then the polynomial
\[{\psi}_{\f}(\lambda) := \sum_{i=0}^d \, (-1)^{i} \dim \Fs_{i}(\f) \, \lambda^{d-i}\]
is equal to the reduced characteristic polynomial $\bar \chi_M(\lambda)$ of $M$.
\end{prop}

\begin{proof}
Consider the dual vector spaces $\Fs^p(\f) := \Fs_p(\f)^*$. Together, the $\Fs^p(\f)$ form a graded algebra with the product induced by the wedge product \cite[Lemma 2]{Zharkov:Bergman}. Moreover, this algebra is naturally isomorphic to the Orlik-Solomon algebra of the matroid $M$ \cite[Theorem 4]{Zharkov:Bergman}. It follows that the dimensions of the graded pieces of this algebra are the coefficients of the reduced characteristic polynomial of $M$, as claimed.  
\end{proof}

It follows from the above proposition that we can recover the weight of the $0$-th dimensional CSM cycle of $M$ from the fan $\f = B(M)$ by setting $\lambda = 1$ in the polynomial ${\psi}_{\f}(\lambda)$. 
Our goal is to recover the weights of the faces in $\csm_k$ in a similar fashion. 

If $\f$ is a polyhedral fan in a vector space $V$ and $\tau$ is a cone of $\f$, the \textbf{star} of 
$\tau$ in $\f$ is the polyhedral fan $\starr_\tau(\f)$ consisting of all cones of the form
\[ \tilde{\sigma} := \{ \lambda(x - z) : \lambda \geq 0, x \in \sigma, \text{and } z \in \tau\}\]
for any cone $\sigma$ of $\f$ containing $\tau$ as a face.
The \textbf{lineality space} of $\f$ is the maximal linear subspace $L$ such that 
$x + L \in |\f|$ for all  $x \in |\f|$.

\begin{lemma}\label{lem:sameweights}
Let $\f = \B(M)$ be the Bergman fan of a loopless matroid $M$.
For any $k$-dimensional cone $\sigma \in \f$ the polynomial $\psi_{{\starr_{\sigma}(\f)}}(\lambda)$
is divisible by $(\lambda -1)^{k-1}$ and
\begin{equation}\label{eqn:sameweights}
 w_{\csm_k(M)}(\sigma)
= \left.\frac{\psi_{{\starr_{\sigma}(\f)}}(\lambda)}{(\lambda -1)^{k-1}}\right|_{\lambda = 1}.
 \end{equation}
\end{lemma}

\begin{proof}
Suppose $\sigma$ is a $k$-dimensional face of $\B(M)$ corresponding to the chain of flats
$\F = \{\emptyset \subsetneq F_1 \subsetneq F_2 \subsetneq \dotsb \subsetneq F_k \subsetneq \{0,\dotsc, n\}\}$ of $M$. 
The fan $\starr_{\sigma}(\f)$ is the Bergman fan of the matroid 
$M_\sigma := \bigoplus_{i=0}^k M|F_{i+1}/F_i$; see for example \cite[Corollary 4.4.8]{MaclaganSturmfels}.
By Proposition \ref{prop:charOS}, the polynomial $\psi_{\starr_{\sigma}(\f)}(\lambda)$ is the reduced characteristic polynomial of the matroid $M_\sigma$.
The characteristic polynomial of $M_\sigma$ is the product of the characteristic polynomial of its components $M_i := M|F_{i+1}/F_i$, hence, 
$$\psi_{\starr_{\sigma}(\f)}(\lambda) = (\lambda -1)^{-1}\prod_{i = 0}^k \chi_{M_i}(\lambda) = (\lambda -1)^{k-1}\prod_{i = 0}^k \bar{\chi}_{M_i}(\lambda).$$ 
As $\bar{\chi}_{M_i}(1) = (-1)^{r(M_i)-1} \beta(M_i)$, in view of Definition \ref{def:chernweight} this directly implies the desired result. 
\end{proof}

We have now all the pieces for the proof of Proposition \ref{thm:csmisomorphism}.

\begin{proof}[Proof of Proposition \ref{thm:csmisomorphism}]
The result follows from Lemma \ref{lem:sameweights} together with the second part of Lemma \ref{lem:invarianceFp}, as the dimensions of the vector spaces $\Fs_p(\f)$ are invariant under invertible linear maps. 
\end{proof}

We present the following lemma which relates the CSM cycles of the different strata of  $\Sigma_M \times \T^r$ and will be useful in future sections. 

\begin{lemma}\label{lem:csmtransbdy}
Let $\sigma$ be a face of a matroid fan $\Sigma_M$ in $\R^n$, and consider 
the corresponding face $\sigma \times \T^r_I$ in $X := \Sigma_M \times \T^r$. Then for subsets $I \subseteq J$ we have 
\[w_{\csm_k(X)}(\sigma \times \T^r_I) = w_ {\csm_{k + |I| -|J|} (X)}(\sigma \times \T^r_J).\] 
\end{lemma}

\begin{proof}
We may suppose that $\dim(\sigma \times \T^r_I) = k$ otherwise the weights on both sides are $0$.
Let $N$ be the matroid corresponding to the matroid fan $\starr_{\sigma}(\Sigma_M)$.
The underlying matroid for the fan $\starr_{\sigma \times \T^r_I}(\Sigma_M \times \T^r_I)$ is $N \oplus C_{r-|I|}$, where $C_{r-|I|}$ is the matroid which is a direct sum of $r - |I|$ coloops. 
The (non-reduced) characteristic polynomials of matroids  are multiplicative under direct summation, so we obtain
$$\chi_{N \oplus C_{r-|I|}} (\lambda) = (\lambda-1)^{r-|I|} \chi_N(\lambda)  = (\lambda-1)^{|J|-|I|} \chi_{N \oplus C_{r-|J|}} (\lambda),$$
since the characteristic polynomial of a single coloop is $\lambda-1$. 
Now the formula for the weights of the CSM cycles follows 
from their description in Lemma \ref{lem:sameweights} and Proposition \ref{prop:charOS}. 
\end{proof}

One could hope that the recipe for the weights of CSM cycles using the dimensions of the vector spaces $\Fs_p(\f)$ produces balanced cycles for any balanced rational polyhedral fan $\Sigma$. However, this is not the case, as the next examples show. 
  
\begin{exa}\label{ex:unbalancedHypersurface}
Consider the fan tropical hypersurface $X \subset \RR^3$ dual to the polytope 
$$P = \text{ConvexHull}\{ (0, 0, 0), (1, 0, 0), (1, 1, 0), (0, 1, 0), (0, 0, 1)\}.$$
The fan $X$ is a $2$-dimensional fan consisting of $5$ rays, with primitive
integer directions $$(0, 0, -1), (-1, 0, 0), (0, -1, 0), (1, 0, 1), (0, 1,
1),$$ and eight $2$-dimensional faces. 
Using the definition of weights from Equation (\ref{eqn:sameweights}) on the $1$-skeleton of $X$ does not produce a balanced tropical cycle. 
If $\sigma$ is any of the $5$ rays of $X$ listed above then 
$$\psi_{{\starr_{\sigma}(\f)}}(\lambda) = \lambda^2 - 3\lambda + 2 = (\lambda-1) (\lambda-2),$$
therefore,  Equation (\ref{eqn:sameweights}) produces $w_{\csm_1(X)}(\sigma) = -1$.
From the above list of directions of these rays
we see that constant weights do not produce a balanced $1$-dimensional fan. 

However, we remark that since $X$ is a tropical hypersurface, \cite{BertrandBihan} give canonical weights on the $1$-skeleton of $X$ that make it balanced. 
\end{exa}
 
\begin{exa}
Consider the $2$-dimensional fan $X$ in $\R^4$ defined in \cite[Section 5.6]{BabaeeHuh}. 
This fan has $14$ rays, in directions
$$ e_1,  \pm e_2, \pm e_3, \pm  e_4, f_1, \pm f_2, \pm f_3, \pm f_4,$$
where $f_1, f_2, f_3, f_4$ are the rows of the matrix, 
$$\begin{array}{rrrr}
0 & 1 & 1 & 1 \\ 
1 & 0 & -1 & 1 \\ 
1 & 1 & 0 & -1 \\ 
1 & -1 & 1 & 0.
\end{array} 
$$
See \cite[Section 5.6]{BabaeeHuh} for the description of the $2$-dimensional faces of $X$. 
Upon calculating the polynomials $\psi_{{\starr_{\sigma}(\f)}}(\lambda)$, we find that  a ray $\sigma$ in a direction 
$e_1, \dots, e_4, f_1, \dots, \text{ or } f_4$ has
$w_{\csm_1(X)}(\sigma) = -1$ whereas a  ray $\sigma$ in direction 
$-e_2, -e_3, -e_4, -f_2, -f_3, \text{ or } -f_4$ has $w_{\csm_1(X)}(\sigma) = 0$. It can be checked that these weights the $1$-skeleton do not satisfy the balancing condition. 
\end{exa}

\begin{exa}
A naive approach to extend the definition of CSM cycles to more general balanced polyhedral fans would be to insist that the additivity property of classical 
CSM classes should also hold in the tropical situation.
Concretely, if the indicator function of a balanced polyhedral fan $\Sigma$ can be expressed as an integer linear combination of indicator functions of matroidal fans, we could try to define the CSM cycles of $\Sigma$ as the corresponding linear combination of the CSM cycles of the matroidal fans. 
However, this does not always produce well-defined tropical cycles, as the sum of the corresponding CSM cycles might depend on the chosen decomposition of $\Sigma$.

As an example, consider the $1$-dimensional fan $\Sigma$ in $\R^2$ with 1 vertex, 6 rays, and support $|\Sigma|=\bigcup_{i=0}^2\spann_\R(\bfe_i)$, 
where $\bfe_1=(1,0)$, $\bfe_2=(0,1)$, and $\bfe_0=(-1,-1)$.
The fan $\Sigma$ is balanced when equipped with weights equal to $1$ on all top dimensional faces.  
The indicator function of $\Sigma$ can be decomposed as a signed sum of indicator functions of matroidal fans in two different ways. 
On the one hand, it decomposes as the sum of $\B(U_{2, 3})$ and $\text{crem}(\B(U_{2, 3}))$ minus the origin $\mathbf{p}$, where $\text{crem} \colon \R^2 \to \R^2$ is the linear map negating both coordinates. 
On the other, it decomposes as the sum of the three lines $\B(M_i) := \spann_{\R}(\bfe _i)$
for $i = 0, 1, 2$, minus two times the origin $\mathbf{p}$. 
\begin{figure}[ht]
\begin{center}
\includegraphics[scale=0.54]{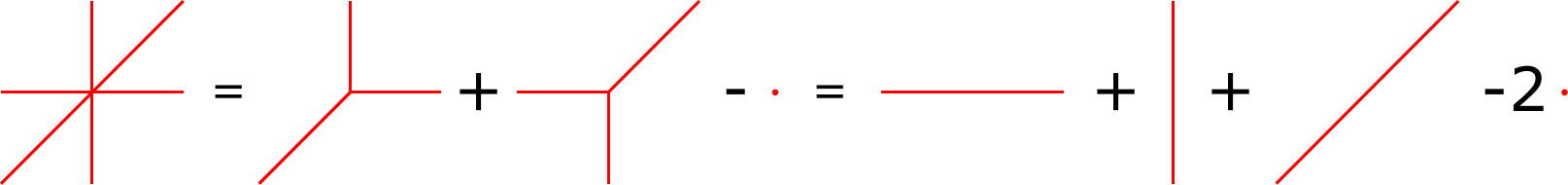}
\caption{Two distinct decompositions of a balanced fan as a (signed) sum of matroidal fans.}
\label{figdecompositions}
\end{center}
\end{figure}
However, the first decomposition yields $\csm_0(U_{2, 3}) + \text{crem}(\csm_0(U_{2, 3})) - \csm_0(U_{1,3}) = (-1-1-1) \mathbf{p} = -3 \mathbf{p}$,
while the second yields $\sum_{i=0}^2 \csm_0(M_i) - 2 \csm_0(U_{1,3}) = (0 + 0 + 0 - 2)\mathbf{p} = -2 \mathbf{p}$.
\end{exa}

\section{Correspondence theorems} 
\label{sec:correspondence}

In this section we establish Theorem \ref{thm:Eulercharacteristic} relating the Euler characteristic of complex algebraic varieties to the degree of $\csm_0$ of their tropicalisation, in the case this tropicalisation is a tropical manifold. 

For $X$ a tropical manifold and $A$ a $0$-dimensional tropical cycle in $X$, let $m_x(A)$ denote the multiplicity of the point $x \in X$ in $A$. 
We denote $$\deg(A) := \sum_{x \in A} m_x(A).$$  
Note that in the case where $A = \csm_0(X)$, 
the sign of the multiplicity $m_x(\csm_0(X))$ depends only on the sedentarity of $x$, 
and thus if $X$ is of dimension $d$ and has only points of sedentarity $0$ 
then $(-1)^d\deg(\csm_0(X)) > 0$.

Suppose  $\mathcal{X}$ is  a meromorphic family of subvarieties of a complex toric variety $\mathcal{Y}$ defined over the punctured disc $ \mathcal{D}^*$. 
 Let $Y$ denote the tropicalization of the toric variety $\mathcal{Y}$ and let $X = \trop(\mathcal{X})$ be the extended tropicalisation of $\mathcal{X} $ in $Y$  in the sense of Payne \cite{Payne}. 
Consider the fibre $\mathcal{X}_t := f^{-1}(t)$ for a generic $t \in \mathcal{D}^*$. 
The following theorem relates the Euler characteristic of $\mathcal{X}_t$ to $\csm_0(X)$ when $X$ is a tropical manifold.

\begin{thm}
\label{thm:Eulercharacteristic}
 Let $\mathcal{X}$ be a meromorphic family of subvarieties of a complex toric variety $\mathcal{Y}$ over the punctured disk $\mathcal{D}^*$. Suppose that $X := \trop(\mathcal{X})$ is a tropical submanifold of $Y := \trop(\mathcal{Y})$, then 
  $$\chi(\mathcal{X}_t) = \deg(\csm_0(X)).$$
\end{thm}
 
\begin{proof}
Consider first the case when  $\mathcal{X}$ is a family of very affine subvarieties, namely is contained in $(\C^*)^N \times \mathcal{D^*}$. 
 Then $\trop(\mathcal{X}) = X \subset \R^N$. Let   ${\Vertices}(\hat{X})$ denote the vertex set of $X$ when equipped with some polyhedral structure $\hat{X}$. 
 By \cite[Corollary 1.4]{KatzStapledon} we have
 $$\chi(\mathcal{X}_t)  = \sum_{v \in {\Vertices}(\hat{X})} \chi(V_{\C}(\In_v \mathcal{X})),$$
 where $\In_v \mathcal{X}$ denotes the initial ideal of $\mathcal X$ at the vertex $v $ of $X$ and $V_{\C}(\In_v \mathcal{X})$ its variety over $\mathbb{C}$.
 
By \cite[Proposition 4.2]{KatzPayne}, since $X$ is a tropical manifold, the initial ideal $\In_v \mathcal{X} \subseteq \mathbb{C}[x_1^{\pm}, \dots, x_N^{\pm}]$ must be generated by linear forms up to toric coordinate change. This means that $V_{\C}(\In_v \mathcal{X})$ is the image under an invertible monomial map of the complement of a hyperplane arrangement. A monomial automorphism of $(\mathbb{C}^*)^n$ is a homeomorphism, and thus the cohomology groups and hence Euler characteristic of the complex points of $V_{\C}(\In_v \mathcal{X})$ are determined by the underlying matroid $M_v$; see  \cite{OrlikTerao}. 
 Namely, the Euler characteristic is up to sign the beta invariant of $M_v$. Therefore we obtain   
  $$\chi(\mathcal{X}_t)  = (-1)^d \sum_{v \in {\Vertices}(\hat X)} \beta(M_v) = 
  \sum_{v \in {\Vertices}(\hat X)} m_v(\csm_0(X)) = \sum_{x \in X} m_x(\csm_0(X)).$$ 
  The last equality holds since $m_x(\csm_0(X)) = 0$ if $x$ lies in a positive-dimensional face of $\hat X$. 
 This proves the statement when $\mathcal{Y}$ is a torus. 
  
Consider now the case where $\mathcal{X}$ is a meromorphic family in 
$\mathcal{Y} \times \mathcal{D}^*$ 
where $\mathcal{Y}$ is some complex toric variety. 
  Let $\Sigma$ be the fan defining $\mathcal{Y}$.  We will apply the above proof for subvarieties of $(\C^*)^N$ to the open strata $\mathcal{X}_{\sigma} := \mathcal{X} \cap (\mathcal{Y}_{\sigma} \times \mathcal{D}^*)$, where $\mathcal{Y}_{\sigma}$ is the open toric strata of $\mathcal{Y}$ corresponding to the cone  $\sigma $ of $\Sigma$. Let $X_{\sigma}$ denote the tropicalisation of $\mathcal{X}_{\sigma}$ in open toric strata
  $\mathcal{Y}_{\sigma}$ corresponding to $\sigma$. 
  By our assumption that $X = \trop(\mathcal{X})$ is a tropical submanifold of $Y = \trop(\mathcal{Y})$, our previous case implies that $\chi( \mathcal{X}_{\sigma,t}) = \deg(\csm_0(X_{\sigma}))$.
  
Compactly supported Euler characteristics are  additive over strata, namely 
$$\chi^c(\mathcal{X}_t) =   \sum_{\sigma \in \Sigma} \chi^c( \mathcal{X}_{\sigma,t}).$$
However, since each stratum is non-singular, Poincar\'e duality relates the usual and compactly supported cohomology groups. Therefore we obtain 
$\chi(V) = (-1)^{\dim_{\mathbb{R}} (V)} \chi^c(V) = \chi^c(V),$ since the real dimensions of the complex varieties are even. 
Hence we have  
$$\chi (\mathcal{X}_t) = \sum_{\sigma \in \Sigma}   \chi( \mathcal{X}_{\sigma,t}) = \sum_{x \in X}  m_x({\csm_0(X)}) = \deg(\csm_0(X)),$$
and the proof is complete. 
 \end{proof}

 We now recall the notion of recession cycle of a tropical cycle, from \cite[Section 5]{AHR}. 
 
 \begin{defi}\label{def:reccycle}
 Let $A$ be a $k$-dimensional tropical cycle in $\R^N$, with weight function $w_A$ on its top-dimensional faces. Its {\bf recession cycle} $\rec(A)$ is a $k$-dimensional fan tropical cycle in $\R^N$ whose support is the recession fan of $A$. The weight function on the $k$-dimensional cones of $\rec(A)$ is given by
 \begin{equation}\label{recweight}
 w_{\rec(A)}(\sigma) = \sum_{\substack{ \sigma' \in A \\ \rec(\sigma')  = \sigma}} w_A(\sigma').
 \end{equation}
 \end{defi} 
 By \cite[Theorem 5.3]{AHR}, the weighted fan $\rec(A)$ is a balanced tropical cycle, and it is tropically rationally equivalent to $A$.

 \begin{rem}\label{rem:EulerChar}
 Under the same notation as in the proof of Theorem \ref{thm:Eulercharacteristic}, if $X$ is a tropical manifold transverse to the boundary of $Y$, for $\sigma $ a cone of $\Sigma$ of dimension $k$ we have
  $$\chi( \mathcal{X}_{\sigma,t}) = \deg(\csm_0(X_{\sigma})) = \sum_{\substack{\alpha \text{ face of } \hat X \\ \dim \alpha = k \\ \rec(\alpha)  = \sigma} } m_{\alpha}(\csm_k(X)).$$
 Indeed, the first equality follows directly from Theorem \ref{thm:Eulercharacteristic} applied to the subvariety $\mathcal{X}_{\sigma}$. 
 Every vertex of $X_{\sigma}$ is contained in the closure of a unique face of $X$ of sedentarity $0$ and dimension $k$ whose recession fan is $\sigma$. 
Moreover, this provides a  bijection between the vertices in $X_{\sigma}$ and such faces of $X$.  
It follows from Lemma \ref{lem:csmtransbdy} that, under this bijection, the weight in $\csm_k(X)$ of  such a $k$-dimensional face is equal to the weight in $\csm_0(X_{\sigma})$ of the corresponding vertex. This proves the above equality. 
 \end{rem}

The next theorem recovers and generalises the correspondence theorem for CSM cycles of matroids and wonderful compactifications proved in  \cite[Theorem 3.1]{LdMRS}. 
Here we consider a subvariety of a toric variety and prove that the CSM class of the subvariety intersected with the big open torus, considered in the Chow ring of the ambient toric variety, is determined by the tropical CSM class. 
Our proof relies heavily on the results and set-up previously established in \cite{Esterov} for tropicalisations of constant families.

The cohomology ring of a non-singular complete toric variety $\mathcal{Y}$ defined by a fan $\Sigma$ is isomorphic to the ring of Minkowski weights on $\Sigma$.  Let $N$ be the dimension of $\mathcal{Y}$, and let $\MW_{k}(\Sigma)$ denote the group of $k$-dimensional Minkowski weights supported on $\Sigma$. 
Then by \cite{FultonSturmfels}, we have $ \MW_{k}(\Sigma)\cong \Ac^{N-k}(\mathcal{Y})$, where $\Ac^{N-k}(\mathcal{Y})$ is the $(N-k)$-th graded piece of the Chow cohomology  ring of $\mathcal{Y}$. 
Since $\mathcal{Y}$ is non-singular and complete, by Poincar\'e duality the cap product with the fundamental class $[\mathcal{Y}]$ provides an isomorphism between Chow homology and cohomology: 
$$\frown [\mathcal{Y}] \colon \MW_k(\Sigma) \xrightarrow{\,\cong\,} \Ac_k(\mathcal{Y}).$$
 
Let   $\mathcal{V}$ be a subvariety of the torus $(\C^*)^N$ and $\mathcal{Y}$ a toric compactification of $(\C^*)^N $ such that  $\overline{\mathcal{V}}$ is transverse to all of the toric  orbits of $\mathcal{Y}$.
Esterov showed that the tropical characteristic classes defined using Euler characteristics recover the Poincar\'e duals of the CSM classes of $\mathcal{V}$ \cite[Theorem 2.39]{Esterov}.
We generalise this result to the tropicalisation of families in the next theorem.  
   
\begin{thm}
\label{thm:corrToric}
Suppose $\mathcal{X}$ is a meromorphic family of subvarieties  of $(\C^*)^N$.   Let $\mathcal{Y}$ be a non-singular projective toric variety 
with defining fan $\Sigma$ in $\R^N$ and $\overline{\mathcal{X}}$ be the closure of $\mathcal{X}$ in $\mathcal{Y}$. 
Suppose that  the closure $\overline{X} = Trop(\overline{\mathcal{X}})$ is  a tropical manifold transverse to the boundary of $Trop(\mathcal{Y})$. Then
$$\csm_k( \mathbbm{1}_{\mathcal{X}_t}) = \rec(\csm_{k}(X)) \frown [\mathcal{Y}] 
 \in \Ac_{k}(\mathcal{Y}).$$
\end{thm}

\begin{proof}
Since $X$ is a tropical manifold and is transverse to the boundary of $\trop(\mathcal{Y})$, all of the open strata  $ \mathcal{X}_{\sigma,t}$ are  sch\"on by \cite[Proposition 7.10]{KatzStapledon}. 
Remark \ref{rem:EulerChar} and  Equation \ref{recweight} imply that for $\sigma$ a $k$-dimensional cone in the recession fan of $X$,  the weight of $\sigma $ in $\rec(\csm_{k}(X))$ is equal to  the Euler characteristic of $ \mathcal{X}_{\sigma,t}$.
 Therefore, the statement follows after applying \cite[Theorem 2.39 (1)]{Esterov}. 
\end{proof}

 \section{Adjunction formula}\label{sec:adjunction}
 
In this section we give a product formula for CSM cycles of matroids, and use it to prove a tropical analogue of the adjunction formula for smooth algebraic varieties. The tropical adjunction formula will require us to restrict to codimension-1 tropical submanifolds in the sense of Definition \ref{def:submanifold}. 

\subsection{Product formula for CSM cycles of matroids}
We start with the following result about intersections of CSM cycles of matroids. If $M$ is a loopless matroid on the ground set $\{0,\dots,n\}$ and $\Sigma_M$ is its corresponding matroidal tropical cycle in $\R^{n+1} / \, \mathbb R \cdot \mathbf{1}$, we consider the formal sum of tropical cycles
$$\csm(\Sigma_M) := \csm(M) := \sum_{i=0}^n \csm_i(M).$$

\begin{prop}\label{prop:intersectionMatroid}
Let $M$ and $M'$ be two loopless matroids on the ground set $\{0,\dots,n\}$, with matroidal tropical cycles $\Sigma_M$ and $\Sigma_{M'}$ in $\R^{n+1} / \, \mathbb R \cdot \mathbf{1} \cong \R^n$. Then
$$\csm(\Sigma_M) \cdot \csm(\Sigma_{M'})= \csm(\Sigma_{M} \cdot \Sigma_{M'}),$$
where the products denote tropical (stable) intersection in $\R^n$.
\end{prop}

\begin{proof}
Any matroidal tropical cycle can be obtained as the intersection of divisors corresponding to codimension-many tropical rational functions \cite{ShawInt}, \cite[Corollary 3.11]{FrancoisRau}. 
In other words, if $c$ and $c'$ are the codimensions of $\Sigma_M$ and $\Sigma_{M'}$ in $\R^n$, respectively, there are tropical rational functions $f_1, \dots, f_c$ and $g_1, \dots, g_{c'}$ on $\R^n$ such that
$$[\Sigma_M] = f_1 \cdot f_2 \cdot \dots \cdot f_c \cdot [\R^n] \quad \text{and} \quad  [\Sigma_{M'}] = g_1 \cdot g_2 \cdot \dots \cdot g_{c'} \cdot [\R^n].$$
This implies that $$[\Sigma_M \cdot \Sigma_{M'}] = f_1 \cdot \dots \cdot f_c \cdot g_1 \cdot \dots \cdot g_{c'} \cdot [\R^n].$$
Using the formula for CSM cycles of matroids in \cite[Proposition 2.3]{RauHopf},
we have 
$$\csm (\Sigma_M) = \prod_{i = 1}^c \frac{1}{1+ f_i} \cdot [\Sigma_M] \quad \text{and} \quad  \csm (\Sigma_{M'}) = \prod_{i = 1}^{c'} \frac{1}{1+ g_i} \cdot [\Sigma_{M'}],$$
and thus
$$\csm (\Sigma_M) \cdot \csm(\Sigma_{M'}) =  \prod_{i = 1}^c \frac{1}{1+ f_i} \prod_{i = 1}^{c'} \frac{1}{1+ g_i} \cdot [\Sigma_M \cdot \Sigma_{M'}] = \csm (\Sigma_M \cdot \Sigma_{M'}),$$
as claimed.
\end{proof}

\subsection{Adjunction in tropical manifolds}

The following theorem relates the CSM cycles of a codimension-$1$ tropical submanifold $D$ to the CSM cycles of the ambient tropical manifold $X$ using the tropical intersection theory in $X$.  
Notice that we do not require that $X$ be compact. 
To prove the theorem we restrict to codimension one tropical  submanifolds from Definition \ref{def:submanifold}.
Recall that Example \ref{ex:singline} presented tropical manifolds (lines) which are also codimension one  tropical subvarieties of tropical manifolds  which do  not satisfy the following theorem, due to the fact that they do not satisfy the compatibility condition to be a tropical submanifold.

\begin{thm}[Adjunction formula]\label{thm:Adjunction}
Let $X$ be a tropical manifold  of dimension $d$ and $D \subset X$ a tropical submanifold of codimension $1$ in $X$. 
Then 
$$\csm_{d-2}(D) = (\csm_{d-1}(X) -  D )\cdot D,$$
where $\cdot$ denotes the tropical intersection product in $X$. 
\end{thm}

\begin{proof}
We assume the submanifold $D$ is connected, otherwise the statement can be proved on each connected component. 
We denote by $D^o$ (respectively, $X^o$) the strata of $D$ (respectively, $X$) of sedentarity $0$.
We first prove the case when $D$ is not a boundary divisor of $X$, in other words, $D$ is the closure in $X$ of $D^o$.

The cycle $D$ is a Cartier divisor in $X$, meaning that there is a collection of tropical rational functions $\{f_{\alpha}\}$ such that $\varphi_\alpha(D \cap U_\alpha) = \divis_{X_{\alpha}}(f_{\alpha})$ for every chart $\varphi_\alpha$. Moreover, in the chart indexed by $\alpha$, the manifold $X$ itself is cut out by a collection of functions $g^{\alpha}_1, \dots, g^{\alpha}_{c_{\alpha}}$. Therefore, in each chart there are collections of functions  $g^{\alpha}_1, \dots, g^{\alpha}_{c_{\alpha}}$ cutting out  $X^o_{\alpha}$, and $f_{\alpha}, g^{\alpha}_1, \dots, g^{\alpha}_{c_{\alpha}}$ cutting out $D_{\alpha} := D \cap U_\alpha$.
Following the proof of Proposition \ref{prop:intersectionMatroid} and working locally in charts we have 
that
\begin{align*}
\csm(D^o_{\alpha}) &= \frac{1}{1+f_{\alpha}} \prod_{i= 1}^{c_\alpha}\frac{1}{1+g^{\alpha}_i} \cdot [D^o_{\alpha}]\\
&= (1-f_{\alpha}+f^2_{\alpha}- \dotsb ) \prod_{i= 1}^{c_\alpha}(1-g^{\alpha}_i+(g^{\alpha}_i)^2- \dotsb )  \cdot [D^o_{\alpha}].
\end{align*}
Looking only at the dimension-$(d-2)$ part of this equation we obtain, 
\begin{align*}
\csm_{d-2}(D^o_{\alpha}) &  = -(f_{\alpha} + g^{\alpha}_1+  \dots +  g^{\alpha}_{c_{\alpha}}) \cdot [D^o_{\alpha}] \\
& = -f_{\alpha} \cdot [D^o_{\alpha}] - (g^{\alpha}_1+  \dots +  g^{\alpha}_{c_{\alpha}}) \cdot [D^o_{\alpha}] \\
& = - (D_{\alpha}^o)^2  + \csm_{d-1}(X^o_{\alpha}) \cdot D^o_{\alpha}.  
\end{align*}
Therefore we can conclude that 
$$\csm_{d-2}(D^o)   = (\csm_{d-1}(X^o) -D^o)\cdot D^o.$$

Notice that, by definition of the boundary of tropical manifolds in terms of sedentarity, when $D$ is of sedentarity $0$ in $X$ we have $\partial D = \partial X \cap D$. 
Since $D$ is a submanifold of $X$ of sedentarity $0$, it is transverse to the  boundary $\partial X$ of $X$.
Moreover, $\partial X$ is a codimension-$1$ tropical cycle when equipped with weights equal to $1$,  
thus as tropical cycles we have 
$\partial D =  \partial X \cdot D$. 
We thus get
\begin{align*}
\csm_{d-2}(D) &= \csm_{d-2}(D^o) + \partial D \\
&= \csm_{d-2}(D^o) + \partial X  \cdot  D \\
&= (\csm_{d-1}(X^o) -D^o) \cdot D^o + \partial X  \cdot  D \\
&= (\csm_{d-1}(X^o) + \partial X)  \cdot  D - (D^o)^2 \\
&= \csm_{d-1}(X)  \cdot  D - D^2.
\end{align*}
The second to last equality holds since $\csm_{d-1}(X^o) \cdot D^o = \csm_{d-1}(X^o) \cdot D$ and $(D^o)^2 = D^2$ due to the transversality of $D$ with the boundary of $X$. This finishes the case when $D$ is not a  boundary divisor of $X$. 

Let $\partial X = \bigcup D_i$ be the decomposition of the boundary of $X$ into irreducible divisors. If $D$ is a boundary divisor of $X$ then
$$\csm_{d-1}(X) = \csm_{d-1}(X^o) + \partial X = \csm_{d-1}(X^o) + D + \sum_{D \neq D_i} D_i.$$   
Therefore
$$\csm_{d-1}(X) \cdot  D - D^2 = \csm_{d-1}(X^o) \cdot D + D \cdot   \sum_{D \neq D_i} D_i.$$
By Lemma \ref{lem:csmtransbdy} we have  
$\csm_{d-1}(X^o) \cdot D = \csm_{d-2}(D^o)$,  
where $D^o=D\setminus\bigcup_{D \neq D_i}D_i$. 
Moreover, the intersection 
$D\cdot \sum_{D \neq D_i} D_i =  \csm_{d-2}(\partial D) = \partial D$.
Combining all this we obtain
$$\csm_{d-1}(X) \cdot  D - D^2 = \csm_{d-2}(D^o) + \csm_{d-2}(\partial D) = \csm_{d-2}(D),$$
which proves the formula. 
\end{proof}

\section{Tropical Noether's Formula}\label{sec:Noether}

In this section we restrict to compact tropical manifolds of dimension $2$, which we call compact tropical surfaces. 
In order to prove Noether's formula for a compact tropical surface $X$,
we will assume the existence of a cellular structure $\hat{X}$ on $X$ satisfying the following properties. 

\begin{defi}\label{def:facestructure}
Let $X$ be a tropical manifold. A cellular structure $\hat{X}$ supported on $X$ is a {\bf face structure} on $X$ if 
for each cell $\sigma$ of $\hat{X}$ there exists 
a chart $\varphi: U_\alpha \to \T^{r_\alpha} \times \R^{n_\alpha}$ of $X$ such that $U_\alpha \supset \sigma$ and $\varphi(\sigma)$ is the closure of a polyhedron in $\R^{r_\alpha} \times \R^{n_\alpha}$ which is transverse to the boundary of $\T^{r_\alpha} \times \R^{n_\alpha}$. 
The face structure $\hat{X}$ is called {\bf rational} if all these polyhedra are $\RR$-rational polyhedra, i.e., their facets have normal vectors with entries in $\QQ$. 

A rational face structure $\hat X$ on a $2$-dimensional tropical manifold is called {\bf Delzant} if for any vertex $v$ of $\hat{X}$ of sedentarity 0 and any $2$-dimensional face $F$ of $\hat{X}$ with $v \in F$ we have that, in any chart of $X$ containing $v$, the primitive integer vectors $\bfw_v(e)$ and $\bfw_v(e')$ of the two edges $e, e'$ of $F$ containing $v$ can be completed to a basis of the ambient lattice. 
\end{defi}

\begin{rem}
The notion of a face structure on a tropical manifold has previously appeared in \cite{JRS}. 
We do not know exactly which tropical manifolds admit face structures, rational face structures, or Delzant face structures. 
However, given a compact tropical surface with a rational face structure $\hat{X}$,  we can use tropical wave front propagation from \cite{MikShk} on each $2$-dimensional  face of $\hat X$ to pass to a finer face structure. This face structure is not necessarily Delzant, but it has the property that for each $2$-dimensional face $F$ and each vertex $v$ of $F$, the outgoing primitive normal vectors to the two edges adjacent to $v$ in $F$ span a triangle with no interior lattice points; see  \cite{MikShk}. 
\end{rem}

We denote the topological Euler characteristic of a tropical manifold $X$ by $\chi(X)$.
Our main goal in this section is to prove the following theorem.  

\begin{thm}[Tropical Noether's Formula]
\label{thm:Noether}
Let $X$ be a compact tropical surface admitting a Delzant face structure.
Then 
$$\chi(X) = \frac{\deg (\csm_0(X) + \csm_1(X)^2) }{12}.$$
\end{thm}

In order to prove Noether's Formula, we first recall formulas for the intersection numbers of $1$-cycles in a compact tropical surface. In a  compact tropical surface, there are  $1$-cycles of sedentarity $0$ and boundary $1$-cycles. A boundary $1$-cycle  is a linear combination of boundary divisors of $X$; see Subsection \ref{sec:boundary}.

\begin{prop}\cite[Section 3.5]{Shaw:Surf}\label{prop:intersections}
Let $X$ be a compact tropical surface and $A, B$ two $1$-cycles in $X$. In the following three cases, the intersection of $A$ and $B$ is defined on the cycle level, as described below. 
\begin{enumerate}
\item If  $A, B$ are both of sedentarity $0$ and transverse to the boundary of $X$, then the intersection $A \cdot B$ is supported on points of sedentarity $0$. Namely, we can write $A\cdot B =  \sum_{x \in (A \cap B)_{0}} m_x(A \cdot B) \, x $. 
\item If $A$ is of sedentarity $0$ and transverse to the boundary of $X$, and $B$
 is an irreducible boundary divisor then $A\cdot B = \sum_{ x\in A \cap B} m_x(A \cdot B) \,x$, where the multiplicity $m_x$ is equal to the weight $w_A(e)$ of the the unique edge $e$ of $A$ containing $x$ in its boundary. 
\item If $A$ and $B$ are distinct irreducible boundary divisors of $X$ then
$A \cdot B = \sum_{x \in A \cap B} x$. 
\end{enumerate}
\end{prop}

We now fix some notation for the rest of this section.
Suppose $X$ is a tropical manifold of dimension $d$. 
For $i = 0,1,\dots, d$, let $X_i$ be the subset of $X$ consisting of points of sedentarity $i$. We let $\hat{X}_i$ denote the subset of cells of $\hat{X}$ whose relative interiors are of sedentarity $i$. 
A cell $\sigma$ of $\hat{X}_0$ is called {\bf bounded} if its closure $\bar{\sigma}$ is contained in $X_0$, and otherwise it is called {\bf unbounded}.

For any polyhedral complex $Y$, we will denote its set of vertices, edges, and $2$-dimensional faces by $\Vertices(Y)$, $\Edges(Y)$, and $\Faces(Y)$, respectively.
Moreover, if $v$ a vertex of $Y$, we denote by $\Edges(v,Y)$ the set of edges of $Y$ containing $v$, and by $\Faces(v,Y)$ the set of $2$-dimensional faces of $Y$ containing $v$.
Similarly, if $e$ is an edge of $Y$, we denote by $\Faces(e,Y)$ the set of $2$-dimensional faces of $Y$ containing $e$.

Let $X$ be a compact tropical surface admitting a Delzant face structure $\hat X$.
Suppose $e$ is an edge of $\hat X$ of sedentarity $0$ and $v$ is a vertex of $e$ of sedentarity $0$.
Fix a chart of $X$ containing the edge $e$.
Since $X$ satisfies the balancing condition with weights equal to $1$ around the edge $e$, there exists an integer $\sigma_v(e)$ such that
\begin{equation}\label{eqn:sigma} 
- \sigma_v(e) \, \bfw_{v}(e) = \sum_{\substack{e' \in {\Edges}(v, \hat X) \\ e' \neq e \text{ are in some face } F \in {\Faces}(v, \hat X)}} \bfw_{v}(e'),
\end{equation}
where $\bfw_{v}(e)$ denotes the primitive integer vector in the direction of $e$ pointing outwards from the vertex $v$. 
Note that the integer $\sigma_v(e)$ is independent of the chart chosen for $X$.

Our strategy for proving Noether's Formula begins by analysing the tropical $0$-cycle 
$\csm_0(X) + \csm_1(X)^2 - \sum_{D \in \partial X} D^2$.
We use the intersection multiplicities in Proposition \ref{prop:intersections} 
to obtain the following presentation of it, supported only at the vertices of $\hat X$.

\begin{lemma}\label{lem:localcontribution}
Let $X$ be a compact tropical surface with a Delzant face structure $\hat{X}$. Let $D_1, \dots, D_k$ denote the irreducible boundary divisors of $X$. Then for any vertex $v \in \hat{X}$ we have 
$$\csm_0(X) + \csm_1(X)^2 - \sum_{i = 1}^k D_i^2 = \sum_{v \in \Vertices(\hat X)} m_v \cdot v$$
with
$$
m_v = 12 - 6|{\Edges}(v, \hat{X})| + 3|{\Faces}(v, \hat{X})| - \delta(v) \cdot \sum_{e \in {\Edges}(v, \hat{X})} \sigma_v(e),
$$
where $\delta(v) = 1$ if $v$ is of sedentarity $0$, and $\delta(v)= 0$ otherwise.
\end{lemma}

\begin{proof}
The cycle $\csm_0(X)$ is supported on the vertices of $\hat{X}$. 
The cycle $\csm_1(X)$ can be written as 
$\csm_1(X) = \csm_1(X_0) + \csm_1(X_1) = \csm_1(X_0) + \sum_{i=1}^k D_i.$ 
Therefore, we have 
\begin{equation}\label{eqn:csm1sq}
\csm_1(X)^2 -  \sum_{i=0}^k D_i^2 = \csm_1(X_0)^2 + 2 \csm_1(X_0) \csm_1(X_1)  + 2\sum_{i < j} D_i D_j. 
\end{equation}

First suppose $v$ is a vertex in $\hat{X}_2$. 
In this case, the multiplicity of $v$ in $\csm_0(X)$ is equal to 1. 
Among the terms on the right-hand side of Equation \eqref{eqn:csm1sq}, the vertex  $v$ only appears with non-zero multiplicity in $2\sum_{i < j} D_i D_j$. Since $v$ is the intersection of two unique boundary divisors $D_i$ and $D_j$ of $X$, we thus have 
$m_v = 3$. 
Moreover, $|{\Edges}(v, \hat{X})| = 2$  and $|{\Faces}(v, \hat{X})|  = 1$. Therefore, in this case we have 
$m_v = 3 = 12 - 6|{\Edges}(v, \hat{X})| + 3|{\Faces}(v, \hat{X})|$,
and the claim holds.

Suppose now $v$ is a vertex in $\hat{X}_1$. The multiplicity of $v$ in $\csm_0(X)$
is equal to $2 - |{\Faces}(v, \hat{X})|$.
The only term on the right-hand side of 
Equation \eqref{eqn:csm1sq} with non-zero multiplicity in $v$ is $2 \csm_1(X_0) \csm_1(X_1)$. Moreover, the cycle  $ \csm_1(X_1)$ is of weight 1 everywhere, and the weight of $ \csm_1(X_0)$ along the unique edge $e$ of $\hat{X}_0$ whose boundary is $v$  is
$2 - |{\Faces}(v, \hat{X})|$ by Lemma \ref{lem:csmtransbdy}.  
Therefore, in this case $m_v = 3(2 - |{\Faces}(v, \hat{X})|)$.
Since each vertex $v \in \hat{X}_1$ is adjacent to exactly one edge in $X_0$ and satisfies $|{\Edges}(v,\hat{X})| = |{\Faces}(v,\hat{X})| + 1$, we can rewrite this as 
$m_v = 3(2 - |{\Faces}(v, \hat{X})|) = 12 -6|{\Edges}(v,\hat{X})| + 3|{\Faces}(v,\hat{X})|$,
and the claim holds.

Lastly, suppose $v \in \hat{X}_0$. The only term on the right-hand side of Equation \eqref{eqn:csm1sq} with non-zero multiplicity in $v$ is
$\csm_1(X_0)^2$. 
Using \cite[Proposition 3.18]{Shaw:Surf}, the term $\csm_1(X_0)^2$ can be expressed as a cycle supported only on the vertices of $\hat X_0$, where a vertex $v \in \hat X_0$ has multiplicity
\begin{equation*}
m_v(\csm_1(X_0)^2) = 10 + N_v - 5|{\Edges}(v, \hat{X})| + 2|{\Faces}(v, \hat{X}) | + \sum_{ e \in {\Edges}(v, \hat{X})} -\sigma_v(e)
\end{equation*}
with $N_v$ denoting the dimension of the affine span of the local fan $\starr_{\hat X}(v)$.

For a vertex $v$ of $\hat{X}_0$, denote by $\Fs_p(v)$ the vector space $\Fs_p( \starr_{\hat X}(v))$ from Definition \ref{def:Fp}.  
By Lemma \ref{lem:sameweights}, the multiplicity of a vertex $v \in \hat X_0$ in the term $\csm_0(X_0)$ is 
$$m_v(\csm_0(X_0)) = \dim \Fs_0(v) - \dim \Fs_1(v) + \dim \Fs_2(v).$$
We have $\dim \Fs_0(v) = 1$ and $\dim \Fs_1(v) = N_v$.
The fan $\Sigma := \starr_{\hat X}(v)$ is matroidal, thus by \cite[Theorem 5.3]{JRS},  it satisfies Poincar\'e duality for tropical homology with $\mathbb{Z}$ and hence $\mathbb{Q}$ coefficients. We refer the reader to \cite[Section 5]{JRS} and Equations 5.1 and 5.3 there for the definitions of the chain complexes computing these groups.  Tropical Poincar\'e duality implies that  
 $H^q(\Sigma; \Fs^p) \cong H^{BM}_{2 - q}(\Sigma; \Fs_{2 - p})$
for all $0 \leq p, q \leq 2$, where the left-hand side denotes tropical cohomology and the right-hand side tropical Borel-Moore homology.  
From their definitions, we have $H^q(\Sigma; \Fs^p) = 0$ for $q \neq 0$ and $\dim H^0(\Sigma; \Fs^p) = \dim \Fs_p(v)$. 
Since  $H^{BM}_{2 - q}(\Sigma; \Fs_{0}) = 0 $ for $q \neq 0$, we can compute its dimension using Euler characteristics.  Notice also that $H^{BM}_{2 }(\Sigma; \Fs_{0}) =H^{BM}_{2 }(\Sigma; \R)$ and that $\dim C^{BM}_q(\Sigma, \R)$ is equal to the number of $q$-dimensional faces of the fan $\Sigma$. 
It follows that 
\begin{align*}
\dim \Fs_2(v) &= \dim H^{BM}_2(\Sigma; \Fs_0) \\
 & = \dim C^{BM}_0(\Sigma, \R)  - \dim C^{BM}_1(\Sigma, \R) + \dim C^{BM}_2(\Sigma, \R) \\
  & = 1 - |\Edges(v, \hat{X})| +  |\Faces(v, \hat{X})|.
\end{align*}
We conclude that the multiplicity of the vertex $v \in \hat X_0$ in the cycle $\csm_0(X_0)$ is equal to  
$$m_v(\csm_0(X))  = 2 - N_v - |{\Edges}(v,\hat{X})| + |{\Faces}(v,\hat{X})|.$$ 
Combining this with the contribution of $m_v(\csm_1(X_0)^2)$ from above, we obtain
$$m_v = 12 - 6|{\Edges}(v, \hat{X})| + 3|{\Faces}(v, \hat{X})| - 
\sum_{e \in {\Edges}(v, \hat{X})} \sigma_v(e),$$
and the claim is proved for all vertices of $\hat{X}$. 
\end{proof}

Following  Lemma \ref{lem:localcontribution}, it remains to take the self-intersection of the boundary divisors of $X$ into account. 
If $D$ is an irreducible  boundary divisor of a  tropical surface $X$, then the self-intersection $D^2$ is only defined up to  rational, homological, or numerical equivalence. 
We are only interested in the self-intersection number, so any of these equivalences will do upon taking the degree. We abuse notation and let  $D^2$ denote the degree of the self-intersection. This degree can be computed in a variety of ways. Here we will do this by finding a section of the normal bundle of $D$ in $X$ and computing its degree \cite{Shaw:Surf}. The section we will use will be {\bf Delzant} in the sense of Definition \ref{def:admissible}. This will not only help us to compute the degree of the normal bundle, but it will be used to construct smooth toric surfaces. 

We briefly review the theory of line bundles, sections, and divisors on tropical curves. 
A tropical curve  $C$ is a tropical manifold of dimension $1$. Here we will assume that $C$ is compact. 
Each boundary divisor $D$ of a compact tropical surface $X$ is a compact tropical curve; an atlas for $D$ is obtained by restricting the atlas for $X$. Moreover, any face structure $\hat X$ on $X$ induces a face structure on $D$.

A line bundle on a tropical curve $C$ is a $2$-dimensional tropical surface $L$ together with a map $\pi : L \to C$ such that for any point $p \in C$ we have $\pi^{-1}(p) = \T$ and furthermore, there exist local trivialisations: for any point $p \in C$ there exists a neighbourhood $U_p$ of $p$ such that  $\pi^{-1}(U_{p}) = U_{p} \times \T$.  If $\{U_{\alpha}, \phi_{\alpha}\}$ is a covering such that $\pi^{-1}(U_{\alpha}) = U_{\alpha} \times \T$ then the line bundle $L$ can be specified via integer affine transition functions $f_{\alpha \beta} : U_{\alpha} \cap U_{\beta} \to \R$ which obey the cocycle condition. 
Let $\mathcal{O}^*$ denote the sheaf of tropical invertible regular functions on $C$. Then line bundles are in correspondence with $H^1(C ; \mathcal{O}^*)$ \cite{MikZha:Jac}. 
The normal bundle $N_X(D)$ of a boundary divisor $D$ in a tropical surface $X$ is the tropical line bundle on $D$ where the transition functions $f_{\alpha \beta} : U_{\alpha} \cap U_{\beta} \to \R$ are inherited from the transition functions from the charts of $X$ in a neighbourhood of $D$ \cite[Section 3.5.4]{Shaw:Surf}.

A section of a line bundle is a continuous function $s : C \to L$ such that $\pi(s(p)) = p$ for all $p \in C$ and in every trivialisation $U \subset C$ with $\pi^{-1}(U) = U \times \T$ the function $s|_U$ is a tropical rational function $U \to \T$ which is bounded. In particular, in a neighbourhood of  each point $p$ of $C$, the section $s$ is given by a piecewise integer affine  function. 
For each $p \in C$ of sedentarity $0$,  we can consider the collection of outgoing primitive integer tangent vectors of $C$ at $p$, call these $\bar{{\bf w}}_p(e_1), \dots, \bar{{\bf w}}_p(e_{{\rm val}(p)})$, where the $e_i$ denote the edges of $C$ adjacent to $p$ and  ${\rm val}(p)$ denotes the valency of $p$ in the underlying graph of $C$. 
By the balancing condition for $C$ at $p$ we have  that $\sum_{i = 1}^{{\rm val}(p)} \bar{{\bf w}}_p(e_i) = 0$. 
Let $\frac{\partial s}{\partial  \bar{{\bf w}}_p(e_i)}(p)$ denote the slope of the integer affine function at $p$ in the direction of $ \bar{{\bf w}}_p(e_i)$. The order of vanishing $m_p(s)$ of $s$ along $p$ is the order of vanishing of the tropical rational function determining $s$ in a local trivialisation at $p$; more specifically, $$m_p(s) = \sum_{i = 1}^{{\rm val}(p)}\frac{\partial s}{\partial \bar{{\bf w}}_p(e_i)} (p).$$
The {\bf  divisor associated to the section} $s$ is $(s) := \sum_{p \in C} m_p(s) p$. 

At a point $p$ of $C$ we will also consider the primitive outgoing tangent vectors to the graph of the section $s$ at the point $s(p)$; namely, the vectors ${{\bf t}}_p(e_1), \dots, {{\bf t}}_p(e_{\rm val(p)})$
given by 
\begin{equation}\label{eq:tangentvectors}
{{\bf t}}_p(e_i) = \left( \bar{{\bf w}}_p(e_i),  \frac{\partial s}{\partial  \bar{{\bf w}}_p(e_i)}(p)\right).
\end{equation}

There is an equivalence between divisors on a tropical curve $C$ and tropical line bundles on $C$ together with a section (up to adding scalars) \cite[Proposition 4.6]{MikZha:Jac}.
Two sections  of the same line bundle that differ only by adding a constant yield the same divisor.

A global tropical rational function on a curve $C$ is a function $h : C \to \T$ such that in each chart of $C$ the function $h$ can be expressed as the difference of two tropical polynomials. We can think of a tropical rational function $h$ as a section of the trivial line bundle $h :C \to \T$.  A divisor is {\bf principal} if it arises from a global tropical rational function $h: C \to \T$ considered as such a section. Rationally equivalent divisors differ by a principal divisor and thus have the same degree. 

If $D$ is a divisor in a compact tropical surface $X$ and $s_D$ is a section of its normal bundle, then by \cite[Section 3.5.4]{Shaw:Surf} we have 
\begin{equation}\label{eq:boundaryselfintersection}
D^2 = \deg((s_D)) = \sum_{p \in D} m_p(s_D).
\end{equation}
To obtain the degree of the self-intersection of a boundary divisor $D$, any section $s_D$ of the normal bundle will do. However, for our proof of Noether's Formula, we will require the notion of a Delzant section. 

\begin{defi}\label{def:admissible}
Let $C$ be a tropical curve with  a face structure $\hat{C}$, and let $L$ be a tropical line bundle on $C$. A section $s : C \to L $ is  {\bf Delzant} if its associated divisor $Q = (s)$ is  supported on the bounded edges of $\hat C$, and for every bounded edge $e$ of $\hat{C}$ there is at most one point $p \in {\rm relint}(e) $ such that $p \in {\rm Supp}(Q)$ and if this point exists we have  $m_p(Q) = 1$. 
\end{defi}

We call such sections Delzant as for every bounded edge $e$ of $\hat C$ and every point $p \in {\rm relint}(e)$, the outgoing tangent vectors to the graph of $s$ are either parallel (in the case $m_p(Q) = 0$) or they are unimodular (in the case $m_p(Q) = 1$). 
 In particular, if $s$ is a Delzant  section, any points with negative multiplicity in the associated divisor $Q = (s)$ must be vertices of $\hat{C}$.

\begin{lemma}\label{lem:admissiblesection} 
For any tropical curve $C$ with face structure $\hat{C}$ and any line bundle $L$ on $C$ with local trivialisations over the edges of $\hat{C}$ there exists a Delzant section $s : C \to L$. 
\end{lemma}

\begin{proof}
By \cite[Proposition 4.6]{MikZha:Jac}, there exists a section $\tilde{s}: C \to L$ of $L$. Denote its associated divisor by $\tilde{Q}$. 
Subtracting any principal divisor  $P$ from $\tilde{Q}$ provides a linearly equivalent divisor and thus another section of $L$. 

We now construct for every edge $e$ of $\hat C$ a certain tropical rational function $h_e : C \to \T$ which is constant outside of $e$. 
For $e$ an unbounded edge of $C$, let $v$ denote the unique vertex of $\hat{C}$ which is of sedentarity $0$. Then we define  $h_e(p) = - \tilde{s}(p) $  if $p $ is a point of $e$, and $h_e(p) = -\tilde{s}(v) $ otherwise. Then $h_e : C \to \T$ is a tropical rational function,  and its associated divisor $(h_e)$ satisfies $m_p((h_e)) = - m_p(\tilde{Q})$ if $p \in {\rm relint}(e)$ or if $p$ is the vertex of $e$ of sedentarity 1, and $m_p((h_e)) = 0$ for all $p \notin e$. 

For $e$ a bounded edge of $\hat{C}$, let $v_1$ and $v_2$ denote its two vertices. 
 If $\tilde{s}(v_1) = \tilde{s}(v_2)$, we   proceed as above and define $h_e(p) = - \tilde{s}(p) $  if $p $ is a point of $e$ and $h_e(p) = -\tilde{s}(v_i)$ otherwise. 
Then $h_e: C \to \T$ is a tropical rational function satisfying $m_p((h_e)) = - m_p(\tilde{Q})$ for all $p \in {\rm relint}(e)$, and $m_p((h_e)) = 0$ for all $p \notin e$. 

Suppose that   $\tilde{s}(v_1) \neq  \tilde{s}(v_2)$. 
Upon taking a chart we can suppose that $v_1, v_2 \in \R$ and that $v_1 < v_2$. Let $$m = \Bigl \lfloor \frac{\tilde{s}(v_2) -\tilde{s}(v_1)}{v_2-v_1} \Bigr\rfloor \quad \text{and} \quad m+1 = -\Bigl\lceil \frac{\tilde{s}(v_2) -\tilde{s}(v_1)}{v_2-v_1} \Bigr\rceil.$$
Let $f_e : \R \to \R$ be the  tropical polynomial given by the piecewise integer affine  function 
$$f_e(x) = \max \{ m(x- v_1) + \tilde{s}(v_1), (m+1)(x-v_2) + \tilde{s}(v_2)\}. $$ 
It can be checked that there is a single zero $p_0$ of $f_e$ in $\R$ and it satisfies $v_1 < p_0 <v_2$. Moreover, the tropical polynomial $f_e$ satisfies $f_e(v_i) = \tilde{s}(v_i)$. Then we can cook up a tropical rational function 
$h_e : C \to L$ by setting $h_e(p) = - \tilde{s}(p) + f_e(p)$ for $p \in e$ and $h_e(p) = 0$ otherwise. Notice that $h_e$ is a tropical rational function and that $m_p((h_e)) = - m_p(\tilde{Q})$ for all $p  \in {\rm relint}(e)$ except $p = p_0$, and $m_p((h_e)) = 0$ for all $p \notin e$.  

By construction, the divisor of the  section  $s = \tilde{s} + \sum_{e \in \Edges(C)} h_e$ is Delzant  in the sense of Definition \ref{def:admissible}, and the lemma is proven. 
\end{proof}

Returning to our proof of Noether's Formula, for each irreducible boundary divisor $D$ of $X$ we fix a Delzant section $s_D$.
Using the Delzant face structure $\hat X$ on $X$ and the choice of Delzant sections $s_D$, 
we will associate to each face $F$ of $\hat X$ a $2$-dimensional rational complete unimodular fan $\Sigma_F$, called the {\rm face fan} of $F$.
  
First, suppose $P$ is a $2$-dimensional $\RR$-rational convex polyhedron in $\R^n$ equipped with an orientation. Let $T(P) \cong \R^2$ denote its tangent space. We define the {\bf edge fan}  $\Sigma_P$ of $P$ to be the  rational polyhedral fan in $T(P)$ whose rays are generated by the (oriented) primitive integer directions $\bfw_P(e)$ of the edges $e$ of $P$ and whose $2$-dimensional cones correspond to the vertices of $P$. See Figure \ref{fig:boundedSigmaF}. 
Notice that choosing another orientation of $P$ simply produces the negative fan to ${\Sigma}_P$. The edge fan ${\Sigma}_P$ is complete if and only if $P$ is compact, and ${\Sigma}_P$ is unimodular if and only if $P$ is Delzant. 

\begin{figure}
\includegraphics[scale=0.7]{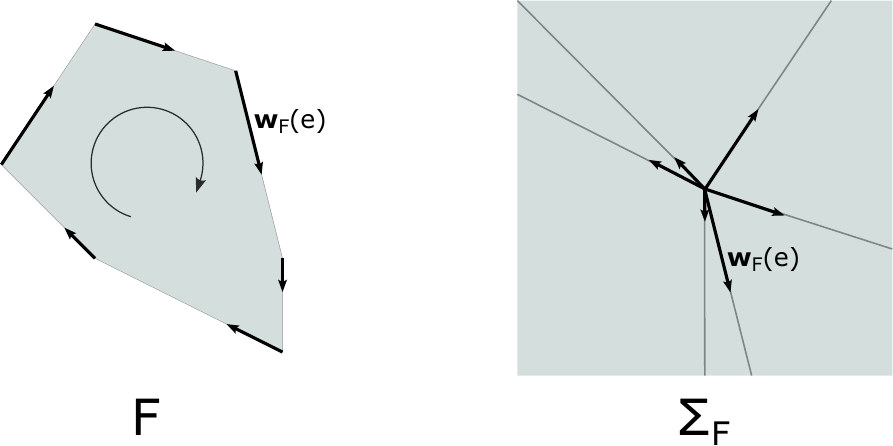} 
\caption{An bounded face $F$ of $\hat{X}$ and its corresponding fan $\Sigma_F$.}
\label{fig:boundedSigmaF}
\end{figure}

\begin{defi}\label{def:facefan}
Let $X$ be a compact tropical surface with a Delzant face structure $\hat{X}$, 
and fix a Delzant  section $s_D$ for each boundary divisor $D$ of $X$; see Lemma \ref{lem:admissiblesection}. 
For every $2$-dimensional face $F$ of $\hat{X}$, we define the {\bf face fan} $\Sigma_F$ of $F$,
which is a $2$-dimensional rational complete unimodular fan. As in Definition \ref{def:facestructure}, choose a chart $\varphi_\alpha$ under which $F$ is the closure of a $2$-dimensional polyhedron $P_F$ in $\R^{n_\alpha} \times \R^{r_\alpha}$ transverse to the boundary of $\T^{n_\alpha} \times \R^{r_\alpha}$, and fix an orientation of $P_F$.

\begin{itemize}
\item If $F$ contains only points of sedentarity 0, then the face fan $\Sigma_F$ is equal to the edge fan $\Sigma_{P_F}$. See Figure \ref{fig:boundedSigmaF}.

\item If $F$ contains points of sedentarity 1 and no points of sedentarity 2,  
the face $F$ contains a unique edge $e$ of sedentarity $1$. The edge $e$ is in some boundary divisor $D$, and has endpoints $v_1, v_2 \in \Vertices(\hat X_1)$.  Suppose the orientation of $F$ induces an orientation from $v_1 $ to $v_2$ on $e$. 
Fix a local trivialisation of the normal bundle to $D$ over the edge $e$.
Recall that ${\bf t}_{v_i}(e)$ denotes the primitive outward tangent vector to the graph of $s_D$ over $e$ at $v_i$.
We complete the edge fan $\Sigma_{P_F}$ to the face fan $\Sigma_F$ by adding the rays generated by the 
tangent vectors ${\bf t}_{v_1}(e)$ and $- {\bf t}_{v_2}(e)$, 
and adding the missing $2$-dimensional cones spanned by adjacent rays. 
If ${\bf t}_{v_1}(e) = -{\bf t}_{v_2}(e)$, i.e., if the section
$s_D$ has no zero on $e$, then only one ray is added. 
The fact that $s_D$ is a Delzant section implies that the fan $\Sigma_F$  is unimodular.
See Figure~\ref{fig:sed1SigmaF}.
\begin{figure}
\includegraphics[scale=0.7]{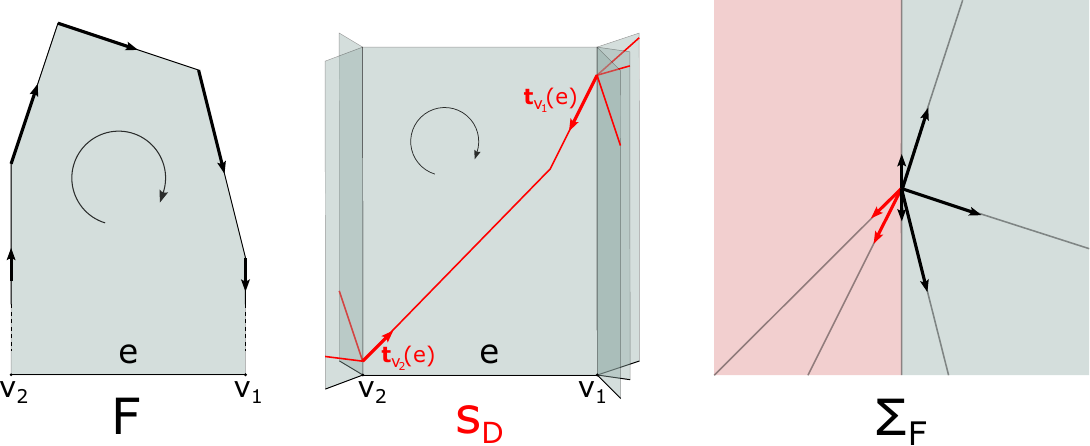} 
\caption{An unbounded face $F$ of $\hat{X}$ with no points of sedentarity 2, the corresponding part of the Delzant section $s_D$ of the boundary divisor $D$, and the resulting fan $\Sigma_F$ associated to $F$. }
\label{fig:sed1SigmaF}
\end{figure}

\item If $F$ contains a point of sedentarity 2, the support of the edge fan ${\Sigma}_{P_F}$ is 
equal to the cone spanned by the (oriented) primitive integer directions of the two unbounded rays of $P_F$.
We construct the face fan $\Sigma_F$ from $\Sigma_{P_F}$ by adding the rays spanned by the negative of these two integer directions, and adding the three $2$-dimensional cones spanned by adjacent rays.
Again, $\Sigma_F$ is a unimodular fan.
See Figure \ref{fig:sed2SigmaF}.
\begin{figure}[b]
\includegraphics[scale=0.7]{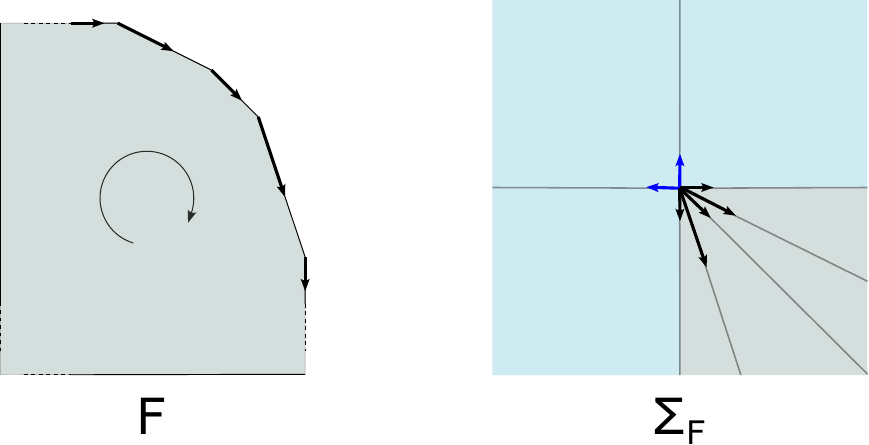} 
\caption{An unbounded face $F$ of $\hat{X}$ with a point of sedentarity 2 and its corresponding fan $\Sigma_F$.}
\label{fig:sed2SigmaF}
\end{figure}

\end{itemize}
Notice that $\Sigma_F$ is defined up to a choice of orientation and chart for $F$. However,  different choices yield a face fan which is ${\rm GL}_2(\Z)$-equivalent. 
\end{defi}

We now relate the geometry of the face fans $\Sigma_F$ to the integers $\sigma_v(e)$ defined in Equation \eqref{eqn:sigma}. Suppose $e$ is an edge of $\hat{X}$ of sedentarity $0$ and $F$ a $2$-dimensional face of $\hat X$ containing $e$. 
Let $\rho_e$ be the ray in the face fan $\Sigma_F$ generated by the (oriented) primitive integer vector $\bfw_F(e)$ of $e$. 
If $\rho$ and $\rho'$ are the two rays in $\Sigma_F$ adjacent to $\rho_e$, denote by ${\bf w}_{\rho'}$ and ${\bf w}_{\rho''}$ their primitive integer vectors. 
Define the integer $\tau_F(e)$ by 
 \begin{equation} \label{eqn:tau} 
 - \tau_F(e) {\bf w}_F(e) = {\bf w}_{\rho'} + {\bf w}_{\rho'}.
\end{equation}
Such an equation always holds for a unique integer $ \tau_F(e)$ since the cones adjacent to the ray $\rho_e$ in $\Sigma_F$ are unimodular \cite[Section 2.5]{Fulton}.

\begin{lemma}\label{lem:ToricInt}
Let $X$ be a compact tropical surface with a Delzant face structure $\hat{X}$. 
If $e$ is a bounded edge of $\hat{X}$ with vertices $v_1, v_2 \in \Vertices(\hat X_0)$, we have 
$$\sigma_{v_1}(e) + \sigma_{v_2}(e)  = - \sum_{\substack{F \in {\Faces}(\hat{X}) \\ F \supset e}} \tau_F(e).$$
If $e$ is an unbounded edge of $\hat{X}$ with vertex $v$ of sedentarity $0$ and vertex $\bar{v}$ 
of sedentarity $1$, we have  
$$\sigma_{v}(e) - m_{\bar{v}}(s_D)  = - \sum_{\substack{F \in {\Faces}(\hat{X}) \\ F \supset e}} \tau_F(e),$$ 
where $D$ is the boundary divisor of $X$ containing $\bar v$ and $s_D$ is the corresponding Delzant section used to construct the face fans in Definition \ref{def:facefan}. 
\end{lemma}

\begin{proof} 
We begin with the case of a bounded edge $e$ with vertices $v_1, v_2 \in \Vertices(\hat X_0)$.
For any (possibly unbounded) $2$-dimensional face $F$ of $\hat X$ adjacent to $e$,
denote by $e_F^1$ and $e_F^2$ the edges of $F$ adjacent to $e$ at the vertices $v_1$ and $v_2$, respectively. 
Without loss of generality, suppose the orientation of $F$ 
induces the orientation on $e$ from $v_1$ to $v_2$.
The rays of $\Sigma_F$ corresponding to $e_F^1$ and $e_F^2$ are then 
$-\bfw_{v_1}(e_F^1)$ and $\bfw_{v_2}(e_F^2)$, respectively, as the orientation induces the inward pointing direction at $v_1$ and the outward pointing direction at $v_2$. See Figure \ref{fig:boundedSigmaF}. Then by Equation \eqref{eqn:tau} we have 
$$- \tau_F(e)\bfw_{v_1}(e)  = -\bfw_{v_1}(e_F^1) + \bfw_{v_2}(e_F^2).$$
Summing over all faces $F$ containing $e$ we get
$$
 - \sum_{\substack{F \in {\Faces}(\hat{X}) \\ F \supset e}} \tau_F(e) \bfw_{v_1}(e)  = 
\sum_{\substack{ e' \in {\Edges}(v_1, \hat{X}) \\ e'\neq e \text{ are in a common face}}} -\bfw_{v_1}(e') + \sum_{\substack{ e'' \in {\Edges}(v_2, \hat{X}) \\e''\neq e \text{ are in a common face}}}  \bfw_{v_2}(e'').
$$
Applying the definition of $\sigma_{v_i}(e)$ from Equation \eqref{eqn:sigma} we obtain 
$$- \sum_{\substack{F \in {\Faces}(\hat{X}) \\ F \supset e}} \tau_F(e) \bfw_{v_1}(e)  = \sigma_{v_1}(e) \bfw_{v_1}(e)  - \sigma_{v_2}(e) \bfw_{v_2}(e).$$ 
Since $\bfw_{v_1}(e)$ is the outward pointing vector of $e$ from $v_1$ and $\bfw_{v_2}(e)$ is the outward pointing vector of $e$ from $v_2$, we have $\bfw_{v_1}(e)  = - \bfw_{v_2}(e)$, 
and so
$$  - \sum_{\substack{F \in {\Faces}(\hat{X}) \\ F \supset e}} \tau_F(e) \bfw_{v_1}(e)  = [\sigma_{v_1}(e) + \sigma_{v_2}(e)] \, \bfw_{v_1}(e),$$
which implies the desired result.

Suppose now $e$ is an unbounded edge of $\hat X$ with vertex $v$ of sedentarity $0$ and vertex $\bar{v}$ of sedentarity $1$.
Consider the Delzant section $s_D$ corresponding to the divisor $D$ containing $\bar v$,
and the local trivialisations of the normal bundle to $D$ used to construct the fans $\Sigma_F$. 
For each edge $\bar e$ of $\hat{D}$ containing $\bar{v}$, recall that ${\bf t}_{\bar{v}}(\bar e)$ denotes the primitive tangent vector to the graph of $s_D$ at $\bar v$ in the local trivialisation over $\bar e$, and that
$$- m_{\bar{v}}(s_D){\bf w}_{v}(e)=  \sum_{\bar{e} \in \Edges(\bar{v}, \hat{D})} {\bf t}_{\bar{v}}(\bar{e})$$
Using this together with Equation \eqref{eqn:sigma}, we can write
$$[ \sigma_v(e) - m_{\bar{v}}(s_D)]\, {\bf w}_{v}(e) = -  \sum_{\substack{e' \in {\Edges}(v, X) \\ e' \neq e \text{ are in some face } F \in {\Faces}(v, X)}} \bfw_{v}(e') + \sum_{\bar{e} \in \Edges(\bar{v}, \hat{D})} {\bf t}_{\bar{v}}(\bar{e}).$$
For any $2$-dimensional face $F$ of $\hat X$ adjacent to $e$ there is a unique edge $e'$ of $F$ of sedentarity $0$ adjacent to $e$ at $v$, and a unique edge $\bar{e}$ of $F$ of sedentarity $1$ adjacent to $e$ at $\bar v$.
We can assume that $F$ was oriented so that the induced orientation on $e$ is from $v$ to $\bar{v}$. 
Then the vector ${\bf t}_{\bar{v}}(\bar{e})$ is exactly the one used to construct the face fan $\Sigma_F$, whereas due to our choice of orientation we have ${\bf w}_F(e') = -\bfw_{v}(e')$. 
By Equation \eqref{eqn:tau} we have 
$$- \tau_F(e) {\bf w}_v(e)  = -{\bf w}_v(e')  + {\bf t}_{\bar{v}}(\bar{e}).$$ 
Adding this over all faces $F$ containing $e$ and comparing with the previous equation proves the claim. 
\end{proof}

The next lemma provides an expression for the number of $2$-dimensional faces of $\hat X$
arising from the Noether's formulas satisfied by the toric surfaces defined by the face fans $\Sigma_F$.

\begin{lemma}\label{NoetherToric}
Let $X$ be a compact tropical surface with a Delzant face structure $\hat{X}$.
Let $D_1, \dots, D_k$ denote the irreducible boundary divisors of $X$.
Then 
$$12|{\Faces}(\hat{X})| = \sum_{v \in {\Vertices}(\hat{X})}  3|{\Faces}(v, \hat{X})| - \sum_{v \in {\Vertices}(\hat{X}_0)}\sum_{e \in {\Edges}(v, \hat{X})} \sigma_v(e) + \sum_{i=1}^k D_i^2.$$
\end{lemma}

\begin{proof}
To prove the lemma, we provide a Noether-type formula for every face $F$ of $\hat X$.
We do this by cases.

$\bullet$ {\em Case $F$ is bounded.} Suppose $F$ is a bounded $2$-dimensional face of $\hat X$. 
As $\hat X$ is a Delzant face structure, the toric variety associated to the fan 
$\Sigma_F$ is a non-singular toric surface. 
Noether's Formula for this toric surface (see \cite[Section 2.5]{Fulton}) says that
\begin{equation}\label{eq:Noetherforpolygon}
12 =  3 |{\Faces}(\Sigma_F)|  +  \sum_{\rho \in {\rm Rays}(\Sigma_F)} D_\rho^2,
\end{equation}
where ${\Faces}(\Sigma_F)$ is the set of $2$-dimensional cones of $\Sigma_F$ and  $D_{\rho}$ is the divisor in the toric surface corresponding to the ray $\rho$ of $\Sigma_F$.  
Since the fan $\Sigma_F$ is unimodular, the self-intersection $D_{\rho}^2$ is equal to $\tau_F(e)$ where $e \in {\Edges}(F)$ is the edge of $F$ corresponding to the ray $\rho$ of $\Sigma_F$ \cite[Section 2.5]{Fulton}. 
Taking into account that all edges of $F$ are of sedentarity $0$, we obtain
\begin{equation*}\label{eq:noethersed0}
12 = 3 |{\Vertices}(F)|  + \sum_{\substack{e \in {\Edges}(F)\\ \sed(e)=0}} \tau_F(e).
\end{equation*}

$\bullet$ {\em Case $F$ is unbounded and contains no vertex of sedentarity 2.} Suppose $F$ is a $2$-dimensional face of $\hat X$ containing points of sedentarity 1 but no points of sedentarity $2$. 
Let $e$ be the unique edge of $F$ of sedentarity 1 and $D_e$ be the boundary divisor containing $e$.
The fan $\Sigma_F$ is unimodular, so the corresponding toric surface satisfies Equation \eqref{eq:Noetherforpolygon}. 
The number of $2$-dimensional cones of $\Sigma_F$ is equal to
$|\Vertices(F)| + 1$ if the section $s_{D_e}$ has a zero on $e$ 
(in which case two extra rays were added to the edge fan of $F$ to obtain $\Sigma_F$), 
and equal to $|\Vertices(F)|$ if $s_{D_e}$ has no zero on $e$ (in which case only one ray was added to the edge fan of $F$).
Moreover, the self-intersections $D_{\rho}^2$ are equal to $\tau_F(e)$ if $\rho$ is a ray of $\Sigma_F$ corresponding to an edge $e \in {\Edges}(F)$ of sedentarity 0, and equal to $-1$ if $\rho$ is one of the extra rays added to the edge fan of $F$.
In all cases, we have
\begin{equation*}\label{eq:noethersed1}
12 = 3 |{\Vertices}(F)| + m_e(s_{D_e}) + \sum_{\substack{e \in {\Edges}(F)\\ \sed(e)=0}} \tau_F(e),
\end{equation*}
where $m_e(s_{D_e}) = 1$ if $s_{D_e}$ has a zero on $e$ and $m_e(s_{D_e}) = 0$ otherwise.

$\bullet$ {\em Case $F$ is unbounded and contains a vertex of sedentarity 2.}
Suppose $F$ is a $2$-dimensional face of $\hat X$ containing a point of sedentarity $2$. 
Again, the fan $\Sigma_F$ is unimodular, and so the corresponding toric surface satisfies Equation \eqref{eq:Noetherforpolygon}. We have $|\Faces(\Sigma_F)|=|\Vertices(F)|$. 
The self-intersections $D_{\rho}^2$ are equal to $\tau_F(e)$ if $\rho$ is a ray of $\Sigma_F$ corresponding to an edge $e \in {\Edges}(F)$ of sedentarity 0, 
and equal to $0$ if $\rho$ is one of the two extra rays added to the edge fan of $F$.
We thus get
\begin{equation*}\label{eq:noethersed2}
12 = 3 |{\Vertices}(F)|  + \sum_{\substack{e \in {\Edges}(F)\\ \sed(e)=0}} \tau_F(e).
\end{equation*}

Now, we add all these equalities over all faces $F$ of $\hat X$ to obtain
$$12|{\Faces}(\hat{X})| =  \sum_{F \in {\Faces}(\hat{X})}  3|{\Vertices}(F)| + \sum_{\substack{e \in \Edges(\hat{X})  \\ \sed(e) = 1}} m_e(s_{D_e}) + \sum_{\substack{F \in \Faces(\hat{X})}} \sum_{\substack{e \in \Edges(F) \\ \sed(e) = 0}}   \tau_{F}(e).$$
Using that $\sum_{F \in {\Faces}(\hat{X})}  |{\Vertices}(F)| = \sum_{v \in {\Vertices}(\hat{X})}  |{\Faces}(v, \hat{X})|$ and changing the order of the summation, we get
$$12|{\Faces}(\hat{X})| =  \sum_{v \in {\Vertices}(\hat{X})}  3|{\Faces}(v, \hat{X})| +  \sum_{\substack{e \in \Edges(\hat{X})  \\ \sed(e) = 1}} m_e(s_{D_e}) + \sum_{\substack{e \in \Edges(\hat{X}) \\ \sed(e) = 0}} \sum_{\substack{F \in \Faces(\hat{X}) \\ F \supset e}}  \tau_{F}(e).$$
Applying Lemma \ref{lem:ToricInt}, we further obtain
$$12|{\Faces}(\hat{X})| =  \sum_{v \in {\Vertices}(\hat{X})}  3|{\Faces}(v, \hat{X})|+   \sum_{\substack{e \in \Edges(\hat{X})  \\ \sed(e) = 1}} m_e(s_{D_e})  + \sum_{\substack{\bar{v} \in \Vertices(\hat{X}) \\ \sed(\bar{v}) = 1}} m_{\bar{v}}(s_{D_{\bar v}}) - \sum_{v \in {\Vertices}(\hat{X}_0)}\sum_{e \in {\Edges}(v, \hat{X})} \sigma_v(e),$$
where $D_{\bar v}$ denotes the boundary divisor containing the vertex $\bar v$.
Finally, Equation \eqref{eq:boundaryselfintersection} gives
$$\sum_{i = 1}^k D_i^2 = \sum_{\substack{e \in \Edges(\hat{X})  \\ \sed(e) = 1}} m_{e}(s_{D_e})+  \sum_{\substack{\bar{v} \in \Vertices(\hat{X}) \\ \sed(\bar{v}) = 1}} m_{\bar{v}}(s_{D_{\bar v}}),$$
which proves the lemma.  
\end{proof}

We now combine these results together to complete the proof of Noether's Formula.

\begin{proof}[Proof of Theorem \ref{thm:Noether}]
From Lemma \ref{lem:localcontribution} we obtain
\begin{multline}\label{ChernSimp}
\deg \left( \csm_0(X) + \csm_1(X)^2 - \sum_{i = 0}^k D_i^2 \right) = 
\sum_{v \in {\Vertices}(\hat{X})}  \left( 12 - 6|{\Edges}(v, \hat{X})| + 3|{\Faces}(v, \hat{X})| \right) \\ 
- \sum_{v \in {\Vertices}(\hat{X}_0)}\sum_{e \in {\Edges}(v, \hat{X})} \sigma_v(e).
\end{multline}
Adding $\deg(\sum_{i = 0}^k D_i^2)$ to both sides, and given that each edge of $\hat X$ is adjacent to exactly two vertices, we obtain
\begin{multline*}
\deg \left(\csm_1(X)^2 + \csm_0(X) \right) =   12 |{\Vertices}(\hat{X}) | - 12|{\Edges}(\hat{X})| + \sum_{v \in {\Vertices}(\hat{X})}  3|{\Faces}(v, \hat{X})|  \\ 
- \sum_{v \in {\Vertices}(\hat{X}_0)}\sum_{e \in {\Edges}(v, \hat{X})} \sigma_v(e) + \sum_{i = 1}^k D_i^2. 
\end{multline*}
Finally, by applying Lemma \ref{NoetherToric}, we get
$$\deg(\csm_1(X)^2 + \csm_0(X)) =   12 |{\Vertices}(\hat{X}) | - 12|{\Edges}(\hat{X})| + 12|\Faces(\hat{X})| = 12 \chi(X),$$
and the proof is complete. 
\end{proof}

We next present a corollary that proves Noether's Formula for toric compactifications of tropical surfaces in $\R^n$. 
To do this we consider the {\bf combinatorial stratification} $X^c$ of $X$ from \cite[Section 1.5]{MikZha:Jac}. 
Suppose that a tropical surface $X \subset \R^n$ has lineality space of dimension $0$. 
Since
the local fans around all vertices of the combinatorial stratification $X^c$ of $X$ are matroidal fans, the cell structure $X^c$ gives a polyhedral face structure of $X$ in $\R^n$.
The collection of unbounded directions of $X$ can also be stratified using the combinatorial stratification. We denote by $\rec{(X^c)}$ this stratified set; the next corollary assumes that this is a unimodular fan.

\begin{cor}\label{cor:surfaceToricvariety}
Let $X \subset \RR^n$ be a tropical submanifold of dimension 2 with lineality space of dimension $0$ such that $\Sigma = \rec{(X^c)}$ is a unimodular fan. Then the closure $\overline{X} \subset \T\Sigma$ of $X$ in the tropical toric variety $\T\Sigma$ is a compact tropical surface satisfying Noether's Formula. 
\end{cor}

\begin{proof}
Since $\Sigma$ is unimodular, the tropical toric variety  $ \T\Sigma$ is non-singular. Moreover, the closure $\overline{X}$ is compact and the closure of each 2-dimensional face of the combinatorial stratification of $X$ is either bounded or, up to extended  integer linear transformation, equal to the compactification in $\T \times \R^{n-1}$ or $\T^2 \times \R^{n-2}$ of a polyhedron transverse to the boundary. 
Therefore the combinatorial stratification of $\overline{X}$ is a Delzant face structure, in the sense of Definition \ref{def:facestructure}. We can thus apply Theorem \ref{thm:Noether} to obtain the statement. 
\end{proof}

Corollary \ref{cor:surfaceToricvariety} can also be extended to a larger class of 2-dimensional tropical submanifolds of toric varieties by applying tropical blowup operations to the surfaces described above and the fact that Noether's Formula is preserved under these operations by \cite[Theorem 5.1]{Shaw:Surf}.

We conclude with an example showing the meaning of Noether's Formula in the particular case where $X$ is a hypersurface of a tropical toric 3-fold.

\begin{exa}
Suppose that $X_0$ is a non-singular tropical hypersurface in $\R^3$ defined by a tropical polynomial $f$. Let $\Delta$ be the Newton polytope of $f$ and let $\Sigma$ be the normal fan of $\Delta$. Consider the compact tropical surface $X$ obtained as the compactification of $X_0$ in the tropical toric variety $\TT\Sigma$, as in Example \ref{ex:hyper}. 

Topologically, $X$ is a wedge of $|{\rm Int}(\Delta) \cap \Z^3|$ spheres, so that 
$$\chi(X) = 1 + |{\rm Int}(\Delta) \cap \Z^3|.$$
The tropical hypersurface $X$ has a face structure $\hat{X}$ dual to the subdivision of its Newton polytope. This face structure is Delzant. 
We can use Lemma \ref{lem:localcontribution} to determine the local contribution of each vertex to $ \deg(\csm_0(X) + \csm_1(X)^2 - \sum_{i = 1}^k D_i^2).$ We have

$$
m_v(\csm_0(X) + \csm_1(X)^2 - \sum_{i = 1}^k D_i^2)= \begin{cases}
  2 & \text{if } \sed(v) = 0, \\
  - 3 &  \text{if } \sed(v) = 1, \\
3 & \text{if } \sed(v) = 2.  
\end{cases}$$

To compute the self-intersection of the boundary divisors $D_i$ we notice that the boundary divisors are in bijection with the facets of the polytope $\Delta$. Each facet has a normal vector ${\mathbf n}_F$, and two facets which intersect do so along an edge. 
It can be shown that the self-intersection number $D_i^2$ is equal to $\tau_i$, given by
$$- \tau_i \, {\mathbf n}_{F_i} = \sum_{\substack{F' \in {\F}(\Delta) \\ F \neq F'}} L(F \cap F') \, {\mathbf n}_{F'}, $$ where $L(F \cap F') $ denotes the lattice length of the intersection of the faces $F$ and $F'$. In particular, the lattice length is $0 $ if $F \cap F' = \emptyset$ or if $F \cap F'$ is a point.

Using the duality between $\hat{X}$ and the unimodular subdivision of $\Delta$, Noether's formula translates to a formula for the number of interior lattice points of a polytope in terms of its lattice volume $ {\rm Vol}(\Delta) $, the total lattice area of its $2$-dimensional faces ${\rm Area} (\Delta)$, the total lattice length of its edges  ${\rm Peri} (\Delta)$, and the numbers $\tau_i$:

$$12(1 + |{\rm Int}(\Delta) \cap \Z^3|) = 2 {\rm Vol}(\Delta) - 3 {\rm Area} (\Delta) + 3 {\rm Peri} (\Delta)   + \sum_{i = 1}^k \tau_i.$$
\end{exa}
We conjecture a generalization of Theorem \ref{thm:Noether} for compact tropical manifolds of any dimension using the tropical Todd class. 
The Todd class of a complex vector bundle $F$ is a formal power series in its Chern classes. The first few terms of the Todd class are 
\[ \Todd =  1 + \frac{c_1}{2} + \frac{c_1^2  + c_2}{12} + \frac{c_1c_2}{24} + \dots.\]
We define the Todd class of a tropical manifold $X$ of dimension $n$ to be obtained from the formal power series above with the substitution 
$c_k = \csm_{n-k}(X)$, which accounts for the indexing of Chern classes by codimension and the CSM classes by dimension. Let $\Todd_d$ denote the degree $d$ part of $\Todd$ above. 

\begin{conj}
For a compact tropical manifold $X$ of dimension $n$ we have 
$$\chi(X) = \deg \Todd_n(X). $$
\end{conj}

\vspace{2mm}

\bibliographystyle{amsalpha}
\bibliography{chernbib}

\end{document}